\documentclass[twoside,a4paper,reqno,11pt]{amsart}
\usepackage[top=28mm,right=30mm,bottom=28mm,left=30mm]{geometry}

\usepackage{amsfonts, amsmath, amssymb, mathrsfs, bm, latexsym, stmaryrd, array, hyperref, mathtools, bold-extra, xcolor, tabularx, booktabs}

\renewcommand{\O}{\Omega}

\newcommand{\normeq}{\trianglelefteqslant}

\newcommand{\la}{\langle}
\newcommand{\ra}{\rangle}

\renewcommand{\to}{\rightarrow}

\newcommand{\leqs}{\leqslant}
\newcommand{\geqs}{\geqslant}

\newcommand{\vs}{\vspace{2mm}}
\newcommand{\fpr}{\mbox{{\rm fpr}}}

\newcommand{\Aut}{\mathrm{Aut}}

\newcommand{\Out}{\mathrm{Out}}
\newcommand{\soc}{\mathrm{soc}}
\newcommand{\PSL}{\mathrm{PSL}}
\newcommand{\GL}{\mathrm{GL}}
\newcommand{\PGL}{\mathrm{PGL}}
\newcommand{\SL}{\mathrm{SL}}

\newcommand{\PGammaL}{\mathrm{P\Gamma L}}
\newcommand{\PSigmaL}{\mathrm{P\Sigma L}}

\newcommand{\PSU}{\mathrm{PSU}}
\newcommand{\POmega}{\mathrm{P\Omega}}

\newcommand{\fix}{\mathrm{fix}}

\newcommand{\Spec}{\mathrm{spec}}

\newcommand{\Tr}{\mathrm{Tr}}

\def\Ma{{\rm M}}
\def\Sym{\mathrm{Sym}}

\makeatletter
\newcommand{\imod}[1]{\allowbreak\mkern4mu({\operator@font mod}\,\,#1)}
\makeatother

\newtheorem{theorem}{Theorem}
\newtheorem*{conj*}{Conjecture}

\newtheorem{conj}[theorem]{Conjecture}
\newtheorem{corol}[theorem]{Corollary}

\newtheorem{thm}{Theorem}[section]
\newtheorem{prop}[thm]{Proposition}
\newtheorem{lem}[thm]{Lemma}
\newtheorem{cor}[thm]{Corollary}

\newtheorem*{prob*}{Problem}

\theoremstyle{definition}
\newtheorem{rem}[thm]{Remark}

\newtheorem*{definition}{Definition}

\newtheorem{ex}[thm]{Example}

\begin{document}
	
	\title[Fixers and derangements of finite permutation groups]{Fixers and derangements of finite permutation groups}
	\author{Hong Yi Huang}
	\address{H.Y. Huang, School of Mathematics, University of Bristol, Bristol BS8 1UG, UK}
    \curraddr{{\sc School of Mathematics and Statistics, University of St Andrews, St Andrews KY16 9SS, UK}}
	\email{hy.huang@bristol.ac.uk}
	\author{Cai Heng Li}
	\address{C.H. Li, SUSTech International Center for Mathematics, and Department of Mathematics, Southern University of Science and Technology, Shenzhen 518055, Guangdong, P. R. China}
	\email{lich@sustech.edu.cn}
	\author{Yi Lin Xie}
	\address{Y.L. Xie, Department of Mathematics, Southern University of Science and Technology, Shenzhen 518055, Guangdong, P. R. China}
	\email{11930521@mail.sustech.edu.cn}
	\date{\today}

	\begin{abstract}
		Let $G\leqs\Sym(\Omega)$ be a finite transitive permutation group with point stabiliser $H$. We say that a subgroup $K$ of $G$ is a fixer if every element of $K$ has fixed points, and we say that $K$ is large if $|K| \geqs |H|$. There is a special interest in studying large fixers due to connections with Erdős-Ko-Rado type problems. 
		In this paper, we classify up to conjugacy the large fixers of the almost simple primitive groups with socle $\PSL_2(q)$, and we use this result to verify a special case of a conjecture of Spiga on permutation characters. We also present some results on large fixers of almost simple primitive groups with socle an alternating or sporadic group.
	\end{abstract}
	
	\maketitle
	
	\section{Introduction}
	
	\label{s:intro}

	Let $G\leqs\mathrm{Sym}(\Omega)$ be a non-trivial finite transitive permutation group with point stabiliser $H$. Recall that a \textit{derangement} for $G$ is a fixed-point-free permutation. This concept has been widely studied for many decades (for example, see the books \cite{BG_classical,GM_EKR}). A classical theorem of Jordan \cite{J_der} shows that $G$ contains a derangement. Of course, subgroups of $G$ may not contain such elements and this leads us naturally to the following definition.
	
	\begin{definition}
		\label{def:fixer}
		A subgroup $K$ of $G$ is a \textit{fixer} if every element in $K$ has fixed points on $\Omega$.
	\end{definition}
	
	\def\Spec{{\sf Spec}}
	
	There are two main sources of motivation for studying fixers, which arise in finite group theory on the one hand and extremal combinatorics on the other. Let us briefly explain these connections. 
	
	As defined in \cite{LPSX_intersecting}, a subset $R$ of $G$ is called {\em semiregular} if $xy^{-1}$ is a derangement for any $x,y\in R$ (in addition, if $R$ is transitive, then we say that it is {\em regular}). The concept of a fixer is complementary to the definition of a semiregular subset. More precisely, \cite[Proposition 1.1]{LPSX_intersecting} states that if $R$ is a semiregular subset and $K$ is a fixer, then $|R||K|\leqs |G|$, with equality if and only if $G = RK$. From this factorisation, one can obtain a lower bound on the proportion of derangements in $G$ by determining the orders of fixers. This is an intensively studied problem, and we refer the reader to \cite{CC_der,FG_der_18,GIP_der} for various results on the proportion of derangements.
	
	In a different direction, recall that $S\subseteq G$ is an {\it intersecting subset} if for any pair of elements $g,h\in S$, the product $gh^{-1}$ is not a derangement. Note that every fixer of $G$ is an intersecting subset. A well-known problem is to determine whether or not the size of an intersecting subset of $G$ is bounded above by the order of $H$, which is motivated by the famous Erdős-Ko-Rado (EKR) theorem \cite{EKR_EKR} in extremal set theory. More specifically, there has been an intense focus on studying the finite transitive permutation groups for which every intersecting subset has size at most $|H|$ (a group with this property is said to have the \textit{EKR property}). This problem dates back to work of Frankl and Deza \cite{FD_EKR} in 1977, and there have been several major advances in this direction in recent years, see \cite{AM_EKR,EFP_EKR,LPSX_EKR,MR_EKR,MST_2-trans} for more details.
	
	Given this connection, we focus our attention on the fixers of large order.
	
	\begin{definition}\label{def:max-fixer}
		Let $K \leqs G$ be a fixer.
		\begin{itemize}\addtolength{\itemsep}{0.2\baselineskip}
			\item[{\rm (i)}] $K$ is \textit{stable} if $K \leqs H^g$ for some $g \in G$, otherwise $K$ is \textit{non-stable};
			\item[{\rm (ii)}] $K$ is {\em large} if it is non-stable and $|K| \geqs |H|$.
		\end{itemize}
	\end{definition}
	
	We say that $G$ has the \textit{weak-EKR property} if every fixer has order at most $|H|$, and $G$ has the \textit{strict-weak-EKR property} if it has no large fixer. As noted in \cite{LPSX_intersecting}, most permutation groups do not satisfy the weak-EKR property; indeed, the largest order of a fixer is often significantly larger than $|H|$.
	To measure the deficit,
	%
	%
	%
	%
	the following invariant
	\begin{equation*}
	\rho_0(G) := \max\left\{\frac{|K|}{|H|\sqrt{|\Omega|}}:\mbox{$K \leqs G$ is a fixer}\right\}.
	\end{equation*}
	is defined in \cite{LPSX_intersecting} (where it is denoted $\rho_0(G/\Omega)$). Note that $\rho_0(G)\geqs 1/\sqrt{|\Omega|}$, with equality if and only if $G$ has the weak-EKR property. 
	
	At the other end of this spectrum, we make the following conjecture (recall that a transitive permutation group is \textit{primitive} if a point stabiliser is a maximal subgroup).
	
	\begin{conj}
		\label{conj:rho_0}
		There exists an absolute constant $c$ such that $\rho_0(G) < c$ for all finite primitive permutation groups $G$.
	\end{conj}
	
	It is shown in \cite[Theorem 1.6]{LPSX_intersecting} that for any number $M$, there exist infinitely many imprimitive permutation groups $G\leqs\mathrm{Sym}(\Omega)$ with $\rho_0(G) > M$, so the primitivity condition is necessary. 
	
	In Theorem \ref{t:pre_prim_large}, we will show that if $G$ is primitive, then $G$ does not have the weak-EKR property (and hence $\rho_0(G) > 1/\sqrt{|\Omega|}$) only if it is almost simple or of product action type (recall that a group $G$ is  \textit{almost simple} if $G_0\normeq G\leqs \Aut(G_0)$ for some non-abelian simple group $G_0$; here $G_0 = \soc(G)$ is the \textit{socle} of $G$). In fact, by constructing a suitable product action as in \cite[Construction 3.2]{LPSX_intersecting}, one can see that Conjecture \ref{conj:rho_0} is equivalent to the following conjecture.
	
	\begin{conj}
		\label{conj:rho_0_as}
		We have $\rho_0(G) \leqs 1$ for all almost simple primitive groups $G$.
	\end{conj}
	
	Our main theorem takes the first step towards a proof of Conjecture \ref{conj:rho_0_as} by determining the large fixers when $G$ is an almost simple primitive group with socle $\PSL_2(q)$. The relevant table is presented at the end of the paper in Section \ref{s:tab} (in the table, $q = p^f$ for some prime $p$ and $L_0 = L\cap \soc(G)$). See Remark \ref{r:tab:PSL2} for additional comments on Table \ref{tab:PSL2}.
	
	\begin{theorem}
		\label{thm:PSL2}
		Suppose $G$ is an almost simple primitive group with point stabiliser $H$ and socle $\PSL_2(q)$. Let $K$ be a subgroup of $G$ with $|K|\geqs |H|$. Then $K$ is a non-stable fixer if and only if $K$ is $G$-conjugate to a subgroup of $L$ described in Table \ref{tab:PSL2}.
	\end{theorem}
	
	Let us briefly outline  the main techniques we use to establish Theorem \ref{thm:PSL2}. First, it is not difficult to handle the case where $G = \PSL_2(q)$ is simple, and this reduces the proof to four special cases listed in Proposition \ref{p:PSL2_summarise}. For each of these cases, we need to study the conjugacy classes of subgroups and elements of $\PGammaL_2(q)$. A key ingredient here is to work with the conjugacy classes of $\mathrm{A\Gamma L}_1(q)$, which is isomorphic to a maximal parabolic subgroup of $\PGammaL_2(q)$. This requires some number theoretic techniques and we refer the reader to Section \ref{s:AGammaL} for more details.
	
	In view of Theorem \ref{thm:PSL2}, we are able to classify all the primitive permutation groups with socle $\PSL_2(q)$ and the (strict-)weak-EKR property. The groups with $q\leqs 61$ can be handled easily by inspecting Table \ref{tab:PSL2}, so we only record the result for $q > 61$. As in \cite{KL_classical}, the \textit{type} of $H$ gives an approximate description of the structure of $H$.
	
	\begin{corol}
		\label{cor:PSL2_weak}
		Suppose $G$ is a primitive permutation group with point stabiliser $H$ and socle $G_0 = \PSL_2(q)$, where $q > 61$. Then $G$ has the weak-EKR property if and only if one of the following holds:
		\begin{itemize}\addtolength{\itemsep}{0.2\baselineskip}
			\item[{\rm (i)}] $H$ is of type $P_1$, $\GL_1(q^2)$, $2^{1+2}_-.\mathrm{O}_2^-(2)$ or $A_5$;
			\item[{\rm (ii)}] $q$ is odd and $H$ is of type $\GL_1(q)\wr S_2$;
			\item[{\rm (iii)}] $f$ is even, $q = 2^f$, $H$ is a subfield subgroup of type $\GL_2(2^{f/2})$, and $|G:G_0|$ is even.
		\end{itemize}
		Moreover, $G$ has the weak-EKR property, but not the strict-weak-EKR property, if and only if $G = G_0$ is one of the final three cases presented in Table \ref{tab:PSL2}.
	\end{corol}
	
	As another application of Theorem \ref{thm:PSL2}, we establish the best possible upper bound on $\rho_0(G)$ for primitive groups $G$ with socle $\PSL_2(q)$, verifying Conjecture \ref{conj:rho_0_as} in this setting.
	
	\begin{corol}
		\label{cor:rho_0_PSL2}
		Suppose $G$ is a primitive permutation group with socle $\PSL_2(q)$. Then $\rho_0(G) < 1/\sqrt{2}$.
	\end{corol}
	
	It can be deduced from Theorem \ref{thm:PSL2} (also remarked in \cite[Lemma 3.1]{LPSX_intersecting}) that if $q$ is even, $G = \SL_2(q)$ and $H = D_{2(q-1)}$, then $\rho_0(G)\to 1/\sqrt{2}$ as $q\to\infty$. We can also prove that the same bound $\rho_0(G) < 1/\sqrt{2}$ holds for every almost simple group with sporadic socle, even though we cannot determine all the large fixers in every case.
	
	\begin{theorem}
		\label{thm:spo_rho_0}
		Suppose $G$ is a primitive permutation group with socle a sporadic simple group. Then $\rho_0(G) < 1/\sqrt{2}$.
	\end{theorem}
	
	Let $G$ be a finite group, let $H$ be a core-free subgroup and let $D(G,H)$ be the set of derangements in $G$ with respect to its action on the set of cosets $[G:H]$. Note that $K\leqs G$ is a fixer on $[G:H]$ if and only if $D(G,H)\subseteq D(G,K)$. This leads us to the more general problem of determining when there is a containment of derangement sets with respect to two different faithful primitive actions of a finite group, which may be of independent interest. See Section \ref{s:tab} for Table \ref{tab:sporadic}.
	
	\begin{theorem}
		\label{thm:spo}
		Let $G$ be an almost simple sporadic group, and let $H,K$ be core-free maximal subgroups of $G$ such that $|K|\geqs |H|$. Suppose $H^G\ne K^G$. Then $D(G,H)\subseteq D(G,K)$ if and only if $(G,H,K)$ is one of the cases listed in Table \ref{tab:sporadic}.
	\end{theorem}
	
	\begin{theorem}
		\label{thm:alt}
		Let $G$ be an almost simple group with alternating socle, and let $H$ be a maximal subgroup of $G$ that acts intransitively or imprimitively on $\{1,\dots,n\}$. Suppose $K$ is a core-free maximal subgroup of $G$ with $|K|\geqs |H|$ and $K\not\cong H$. Then $D(G,H)\subseteq D(G,K)$ if and only if $(G,H,K) = (A_5,S_3,A_4)$.
	\end{theorem}
	
	We refer the reader to Lemmas \ref{l:max_H_prim_K_imprim}, \ref{l:max_K=HA} and \ref{l:max_K=PA} for partial results in the case where $\soc(G) = A_n$ and $H$ acts primitively on $\{1,\dots,n\}$.
	
	As before, let $G \leqs {\rm Sym}(\O)$ be a finite transitive permutation group with point stabiliser $H$ and define
	\begin{equation*}
	\rho_1(G) := \max\left\{\frac{|K|}{|H|\sqrt{|\Omega|}}:\mbox{$K \leqs G$ is a fixer and $K$ is a maximal subgroup of $G$}\right\}.
	\end{equation*}
	Clearly, we have $\rho_1(G)\leqs \rho_0(G)$, so 
	Conjecture \ref{conj:rho_0} asserts that $\rho_1(G)$ is bounded above by an absolute constant if $G$ is primitive. As noted above, if $q$ is even, $G = \SL_2(q)$ and $H = D_{2(q-1)}$, then $\rho_1(G) = \rho_0(G)\to 1/\sqrt{2}$ as $q\to \infty$. In addition, we anticipate that $\rho_1(G) < 1/\sqrt{2}$ for all almost simple primitive groups. As a corollary of Theorems \ref{thm:spo} and \ref{thm:alt}, we establish the best possible upper bound on $\rho_1(G)$ for primitive groups $G$ with alternating or sporadic socle.
	
	\begin{corol}
		\label{cor:rho_1_An}
		Suppose $G$ is an almost simple primitive group with socle an alternating or sporadic group. Then $\rho_1(G) \leqs \sqrt{2/5}$, with equality if and only if $(G,H) = (A_5,S_3)$.
	\end{corol}
	
	As a special case of the derangement containment problem highlighted above, there is a particular interest in studying the finite groups $G$ with core-free maximal subgroups $H$ and $K$ such that $D(G,H) = D(G,K)$. This problem arises naturally in algebraic number theory. More specifically, it is remarked in \cite[Section 1]{J_number} (also see \cite{K_arith}) that if $F/k$ is a Galois extension containing subfields $F_1$ and $F_2$, $G = \mathrm{Gal}(F/k)$ is the Galois group of $F/k$, $H = \mathrm{Gal}(F/F_1)$ and $K = \mathrm{Gal}(F/F_2)$, then $F_1$ and $F_2$ are \textit{Kronecker equivalent} (their Kronecker sets differ by a finite number of primes) if and only if
	\begin{equation*}
	\bigcup_{g\in G}H^g = \bigcup_{g\in G}K^g
	\end{equation*}
	(if both $H$ and $K$ are core-free in $G$, then the latter condition is equivalent to $D(G,H) = D(G,K)$). See \cite[Section 1.1]{FPS_der} and \cite[Section 4]{P_Kro} for more details about this connection.
	
	The following conjecture has been proposed by Pablo Spiga (see \cite[Conjecture 1]{S_06}).
	
	\begin{conj}[Spiga]
		\label{conj:Spiga}
		Let $G$ be a finite group acting faithfully and primitively on the cosets $[G:H_1]$ and $[G:H_2]$, with respective permutation characters $\pi_1$ and $\pi_2$. If $D(G,H_1) = D(G,H_2)$, then either $\pi_1 = \pi_2$, or one of $\pi_1-\pi_2$ and $\pi_2-\pi_1$ is a character of $G$.
	\end{conj}
	
	In \cite{S_06}, Conjecture \ref{conj:Spiga} is reduced to almost simple groups (see \cite[Theorem 1.2]{S_20} for an explicit statement) and it has been verified for all almost simple sporadic groups \cite[Theorem 1.3]{S_20}. The latter paper also establishes some partial results towards Conjecture \ref{conj:Spiga} for alternating and symmetric groups.
	
	Our final result verifies Conjecture \ref{conj:Spiga} for the groups with socle $\PSL_2(q)$. This is the first family of groups of Lie type for which the conjecture has been resolved.
	
	\begin{theorem}
		\label{thm:Spiga}
		Conjecture \ref{conj:Spiga} holds for all groups with socle $\PSL_2(q)$.
	\end{theorem}
	
	The proof of Theorem \ref{thm:PSL2}, and also Corollaries \ref{cor:PSL2_weak} and \ref{cor:rho_0_PSL2}, will be completed in Section \ref{s:thm1}. In Section \ref{s:Spiga}, we will establish Theorem \ref{thm:Spiga}, while Theorems \ref{thm:spo_rho_0} and \ref{thm:spo} will be proved in Section \ref{s:spo}. Finally, we complete the proofs of Theorem \ref{thm:alt} and Corollary \ref{cor:rho_1_An} in Section \ref{s:max}.

	\subsection*{Notation}
	
	Let $n$ be a positive integer. Then the set $\{1,\dots,n\}$ is sometimes denoted $[n]$. If $p$ is a prime divisor of $n$, then we write $n_p$ for the largest $p$-power dividing $n$, and we also write $n_{p'} = n/n_p$. Similarly, if $G$ is a cyclic group and $p$ is a prime divisor of $|G|$, then the unique Sylow $p$-subgroup and the unique Hall $p'$-subgroup of $G$ are denoted $G_p$ and $G_{p'}$, respectively.
	We adopt the standard notation for simple groups of Lie type from \cite{KL_classical}.

	\subsection*{Acknowledgements}
	
	This work was partially supported by NNSFC grant no. 11931005. The first author thanks the China Scholarship Council for supporting his doctoral studies at the University of Bristol. He also thanks the Southern
	University of Science and Technology (SUSTech) for their generous hospitality during a
	visit in 2023. The authors thank Tim Burness for his helpful comments on an earlier draft of the paper.

	\section{Preliminaries}
	
	\label{s:pre}
	
	Throughout, let $G\leqs\mathrm{Sym}(\Omega)$ be a finite transitive permutation group with point stabiliser $H$.
	
	\subsection{First observations}
	
	\label{ss:pre_obs}
	
	We first give some basic observations of the fixers of permutation groups.
	
	\begin{lem}
		\label{l:pre_iff}
		Let $K$ be a subgroup of $G$. Then the following statements are equivalent.
		\begin{itemize}\addtolength{\itemsep}{0.2\baselineskip}
			\item[{\rm (i)}] $K$ is a fixer.
			\item[{\rm (ii)}] $K^g$ is a fixer for some $g\in G$.
			\item[{\rm (iii)}] $K\subseteq \bigcup_{g\in G}H^g$.
			\item[{\rm (iv)}] $\bigcup_{g\in G}K^g\subseteq \bigcup_{g\in G}H^g$.
			\item[{\rm (v)}] Every element of $K$ is $G$-conjugate to an element in $H$.
			\item[{\rm (vi)}] For any $G$-conjugacy class $C$, $C\cap K\ne \emptyset$ implies that $C\cap H\ne \emptyset$.
		\end{itemize}
	\end{lem}
	
	The \textit{spectrum} of a finite group $X$, denoted $\Spec(X)$, is the set of the orders of the elements of $X$. We also write $\pi(X)$ for the set of prime divisors of $|X|$, so $\pi(X)\subseteq\Spec(X)$.
	
	\begin{lem}
		\label{l:pre_fix}
		Let $K$ be a fixer of $G$. Then
		\begin{itemize}\addtolength{\itemsep}{0.2\baselineskip}
			\item[{\rm (i)}] $\Spec(K)\subseteq \Spec(H)$;
			\item[{\rm (ii)}] $\pi(K)\subseteq \pi(H)$;
			\item[{\rm (iii)}] $K$ is intransitive on $\Omega$; and
			\item[{\rm (iv)}] $KH\ne G$.
		\end{itemize}
	\end{lem}

	\begin{proof}
		Parts (i) and (ii) are clear from Lemma \ref{l:pre_iff}. For parts (iii) and (iv), recall that any finite transitive group has a derangement \cite{J_der}.
	\end{proof}

	\begin{lem}
		\label{l:pre_red}
		Let $K$ be a fixer of $G$ and let $G_0\leqs G$. Then $K\cap G_0$ is a fixer of $G_0$.
	\end{lem}

	Let $\Delta$ be a finite set and suppose $X\leqs\mathrm{Sym}(\Delta)$. Then
	\begin{equation*}
	\fpr_\Delta(x):=\frac{|\fix_\Delta(x)|}{|\Delta|}
	\end{equation*}
	is the \textit{fixed point ratio} of $x\in X$, where $\fix_\Delta(x)$ is the set of fixed points of $x$ on $\Delta$. If $X$ is transitive on $\Delta$, then we have
	\begin{equation*}
	\fpr_\Delta(x) = \frac{|x^X\cap X_\delta|}{|x^X|},
	\end{equation*}
	where $\delta\in \Delta$ (see, for example, \cite[Lemma 1.2(iii)]{B_fpr_survey}).
	
	\begin{lem}
		\label{l:pre_fpr}
		Suppose there exists a set $\Delta$ such that $G\leqs\mathrm{Sym}(\Delta)$ and there exists $x\in K$ such that $\fpr_\Delta(x)\ne \fpr_\Delta(y)$ for any $y\in H$. Then $K$ is not a fixer of $G$.
	\end{lem}
	
	\begin{proof}
		This is given by the fact that $\fpr_\Delta(x) = \fpr_\Delta(x^g)$ for any $g\in \mathrm{Sym}(\Delta)$.
	\end{proof}
	
	Recall that the \textit{minimal degree} $\mu_\Delta(X)$ of $X$ is defined to be the smallest number of points moved by any non-identity element. That is,
	\begin{equation*}
	\mu_\Delta(X) = \min_{1\ne x\in X}\left(|\Delta|-|\fix_\Delta(x)|\right) = |\Delta|\cdot\left(1-\max_{1\ne x\in X}\fpr_\Delta(x)\right).
	\end{equation*}
	
	\begin{cor}
		\label{c:pre_min_deg}
		Suppose $K$ is a fixer of $G$. Then for any set $\Delta$ with $G\leqs\mathrm{Sym}(\Delta)$, we have $\mu_\Delta(K)\geqs \mu_\Delta(H)$.
	\end{cor}
	
	\begin{proof}
		If $\mu_\Delta(K) < \mu_\Delta(H)$, then there exists $x\in K$ such that $\fpr_\Delta(x) > \fpr_\Delta(y)$ for any $y\in H$. Now apply Lemma \ref{l:pre_fpr}.
	\end{proof}

	\subsection{Large fixers}
	
	\label{ss:pre_large}
	
	Recall that a non-stable fixer $K$ is \textit{large} if $|K| \geqs |H|$, and we say $K$ is \textit{strictly large} if $|K| > |H|$.
	
	\begin{ex}
		\label{ex:subfield}
		Let $G=\PGL_n(q)$ and suppose the point stabiliser $H=\PGL_n(q_0)$ comprises the images of all the invertible matrices over $\mathbb{F}_{q_0}$, where $q_0^r = q$ for some integer $r\geqs 3$. Let $B = U{:}D$ be a Borel subgroup of $G$ consisting of the images of all upper-triangular in $\GL_n(q)$, where $U$ is a Sylow $p$-subgroup of $G$ and $D$ is a maximal split torus. Then $K:=U{:}(D\cap H)$ is a large fixer of $G$, and so $G$ does not have the EKR property. To see this, first note that
		\begin{equation*}
		|K| =q^{n(n-1)/2}(q_0-1)^{n-1}\geqs q_0^{n^2-1} > |H|.
		\end{equation*}
		Let $x\in K$ and let $\widehat{x}\in B$ be a pre-image of $x$. Then $\widehat{x}$ is an upper-triangular matrix whose diagonal entries are all in $\mathbb{F}_{q_0}$. Hence, the elementary divisors of $\widehat{x}$ are polynomials in $\mathbb{F}_{q_0}$, and \cite[Theorem 6.7.3]{H_alg} implies that $x$ is $G$-conjugate to an element in $H$.
	\end{ex}
	
	As introduced in Section \ref{s:intro}, we say $G$ has the \textit{weak-EKR property} if $G$ has no strictly large fixer, while $G$ is said to have the \textit{strict-weak-EKR property} if $G$ has no large fixer. Now we list some results in the literature in the study of the EKR properties of permutation groups, which will be helpful in the study of weak-EKR properties, noting that the EKR property implies the weak-EKR property. The first result is \cite[Theorem 1.1]{MST_2-trans}.
	
	\begin{thm}
		\label{t:2-trans}
		Any finite $2$-transitive permutation group has the EKR property.
	\end{thm}
	
	The following is \cite[Lemma 14.6.1]{GM_EKR}.
	
	\begin{lem}
		\label{l:pre_strict-ekr}
		If $G$ has a regular subgroup, then $G$ has the EKR property.
	\end{lem}
	
	\begin{cor}
		\label{c:pre_strict-ekr}
		If $G$ has a regular subgroup, then $G$ does not have a strictly large fixer.
	\end{cor}
	
	Recall that $G$ is called \textit{primitive} if $H$ is a maximal subgroup of $G$.
	
	

	\begin{thm}
		\label{t:pre_prim_large}
		Let $G\leqs\mathrm{Sym}(\Omega)$ be a primitive group without the weak-EKR property. Then $G$ can be embedded in a wreath product $L\wr S_k$, where $k\geqs 1$, $L\leqs \mathrm{Sym}(\Sigma)$ is almost simple, and $\Omega = \Sigma^k$.
	\end{thm}
	
	\begin{proof}
		If $G$ cannot be embedded in $L\wr S_k$ with its product action, then one can check in view of O'Nan-Scott theorem \cite{LPS_O'Nan-Scott} that $G$ has a regular subgroup. Now apply Corollary \ref{c:pre_strict-ekr}.
	\end{proof}
	
	The following elementary lemma will be useful to study the fixers of almost simple primitive groups with large sizes.
	
	\begin{lem}
		\label{l:pre_as_size}
		Suppose $G$ is a primitive group with socle $G_0$, and let $K$ be a subgroup of $G$. Then
		\begin{equation*}
		|K\cap G_0|\geqs \frac{|K|\cdot |H\cap G_0|}{|H|}.
		\end{equation*}
		In particular, if $|K|\geqs |H|$ (resp. $|K| > |H|$) then $|K\cap G_0|\geqs |H\cap G_0|$ (resp. $|K\cap G_0| > |H\cap G_0|$).
	\end{lem}
	
	\begin{proof}
		Note that $|H\cap G_0| = |H|/|G:G_0|$ since $G_0$ is transitive.
	\end{proof}
	
	\begin{cor}
		\label{c:pre_as_size}
		Suppose $G$ is a primitive group with socle $G_0$. Then $\rho_0(G)\leqs \rho_0(G_0)$.
	\end{cor}
	
	\begin{proof}
		Combine Lemmas \ref{l:pre_red} and \ref{l:pre_as_size}.
	\end{proof}

	\section{Conjugacy classes of $\mathrm{A\Gamma L}_1(q)$}
	
	\label{s:AGammaL}
	
	To study the fixers of two-dimensional linear groups in Sections \ref{s:PSL2} and \ref{s:thm1}, we will use the information of conjugacy classes of $\mathrm{A\Gamma L}_1(q)$, noting that it is isomorphic to a maximal subgroup of $\PGammaL_2(q)$ of type $P_1$. Here we write
	\begin{equation}
	\label{e:Gamma}
	\Gamma = (\mathbb{F}_q^+{:}\mathbb{F}_q^\times){:}\la \phi\ra\cong\mathrm{A\Gamma L}_1(q),
	\end{equation}
	where $q = p^f$ for some prime $p$, and $\la\phi\ra = \mathrm{Gal}(\mathbb{F}_q/\mathbb{F}_p)$ is the Galois group, with the multiplication
	\begin{equation*}
	(a,\lambda)\phi^i(b,\mu)\phi^j = (a+\lambda^{-1}b^{\phi^{-i}},\lambda\mu^{\phi^{-i}})\phi^{i+j}.
	\end{equation*}
	Recall that for a subfield $\mathbb{F}_{q_0}$ of $\mathbb{F}_q$, the trace map $\mathrm{Tr}_{\mathbb{F}_q/\mathbb{F}_{q_0}}:\mathbb{F}_q\to\mathbb{F}_{q_0}$ is given by
	\begin{equation*}
	\Tr_{\mathbb{F}_q/\mathbb{F}_{q_0}}(x) = \sum_{\sigma\in\mathrm{Gal}(\mathbb{F}_q/\mathbb{F}_{q_0})}\sigma(x).
	\end{equation*}
	It is well-known that the trace map $\mathrm{Tr}_{\mathbb{F}_q/\mathbb{F}_{q_0}}$ is surjective since $\mathbb{F}_q/\mathbb{F}_{q_0}$ is a separable extension. That is, $\Tr_{\mathbb{F}_q/\mathbb{F}_{q_0}}(\mathbb{F}_q) = \mathbb{F}_{q_0}$.
	
	The following lemma combines \cite[Lemmas 2.14 and 2.15]{Sh_conj_AGammaL}, which describes the conjugacy classes of $\Gamma$.
	
	\begin{lem}
		\label{l:pre_AGammaL_Sh}
		The elements $(a,1)\phi$ and $(b,1)\phi$ of $\Gamma$ are conjugate in $\Gamma$ if and only if $\Tr_{\mathbb{F}_q/\mathbb{F}_{p}}(a)\Tr_{\mathbb{F}_q/\mathbb{F}_{p}}(b)\ne 0$ or $\Tr_{\mathbb{F}_q/\mathbb{F}_{p}}(a) = \Tr_{\mathbb{F}_q/\mathbb{F}_{p}}(b) = 0$.
	\end{lem}
	
	In this section, we extend Lemma \ref{l:pre_AGammaL_Sh} by considering conjugacy classes in $\Gamma$ with some specific elements. Suppose $q = q_1^s = q_0^r$ for some positive integers $s$ and $r$, and assume $r$ is prime. We first record some observations on the trace maps.
	
	\begin{lem}
		\label{l:pre_AGammaL_trace}
		The following properties hold.
		\begin{itemize}\addtolength{\itemsep}{0.2\baselineskip}
			\item[{\rm (i)}] If $\mathbb{F}_{q_1}\subseteq\mathbb{F}_{q_0}$ and $r = p$, then $\Tr_{\mathbb{F}_q/\mathbb{F}_{q_1}}(\mathbb{F}_{q_0})=\{0\}$.
			\item[{\rm (ii)}] If $\mathbb{F}_{q_1}\not\subseteq \mathbb{F}_{q_0}$ or $r\ne p$, then $\Tr_{\mathbb{F}_q/\mathbb{F}_{q_1}}(\mathbb{F}_{q_0}) = \mathbb{F}_{q_1}\cap \mathbb{F}_{q_0}$.
		\end{itemize}
	\end{lem}
	
	\begin{proof}
		First note that if $r = p$ then $\Tr_{\mathbb{F}_q/\mathbb{F}_{q_0}}(\mathbb{F}_{q_0}) = \{0\}$. Thus, if $\mathbb{F}_{q_1}\subseteq\mathbb{F}_{q_0}$ and $r = p$ then
		\begin{equation*}
		\Tr_{\mathbb{F}_q/\mathbb{F}_{q_1}}(x) = \Tr_{\mathbb{F}_{q_0}/\mathbb{F}_{q_1}}(\Tr_{\mathbb{F}_q/\mathbb{F}_{q_0}}(x)) = \Tr_{\mathbb{F}_{q_0}/\mathbb{F}_{q_1}}(0) = 0
		\end{equation*}
		for any $x\in\mathbb{F}_{q_0}$. This gives part (i).
		
		Next, we consider the case where $\mathbb{F}_{q_1}\subseteq\mathbb{F}_{q_0}$ and $r\ne p$. Assume $x\in\mathbb{F}_{q_0}$, noting that $\Tr_{\mathbb{F}_q/\mathbb{F}_{q_0}}(x) = rx$. Thus,
		\begin{equation*}
		\Tr_{\mathbb{F}_q/\mathbb{F}_{q_1}}(x) = \Tr_{\mathbb{F}_{q_0}/\mathbb{F}_{q_1}}(\Tr_{\mathbb{F}_q/\mathbb{F}_{q_0}}(x)) = \Tr_{\mathbb{F}_{q_0}/\mathbb{F}_{q_1}}(rx) = r\Tr_{\mathbb{F}_{q_0}/\mathbb{F}_{q_1}}(x).
		\end{equation*}
		This implies that $\Tr_{\mathbb{F}_q/\mathbb{F}_{q_1}}(\mathbb{F}_{q_0}) = \mathbb{F}_{q_1}$ since $(r,p) = 1$ and $\Tr_{\mathbb{F}_{q_0}/\mathbb{F}_{q_1}}(\mathbb{F}_{q_0}) = \mathbb{F}_{q_1}$.
		
		Finally, let us assume $\mathbb{F}_{q_1}\not\subseteq\mathbb{F}_{q_0}$, so $\mathbb{F}_q = \mathbb{F}_{q_1}\mathbb{F}_{q_0}$ since $\mathbb{F}_{q_0}$ is a maximal subfield of $\mathbb{F}_q$. Hence, if $\Tr_{\mathbb{F}_q/\mathbb{F}_{q_1}}(\mathbb{F}_{q_0}) = \{0\}$, then
		\begin{equation*}
		\Tr_{\mathbb{F}_q/\mathbb{F}_{q_1}}(\mathbb{F}_q) = \Tr_{\mathbb{F}_q/\mathbb{F}_{q_1}}(\mathbb{F}_{q_1}\mathbb{F}_{q_0}) = \{0\},
		\end{equation*}
		a contradiction to the surjectivity of $\Tr_{\mathbb{F}_q/\mathbb{F}_{q_1}}$. It follows that there exists $x\in\mathbb{F}_{q_0}$ such that $\Tr_{\mathbb{F}_q/\mathbb{F}_{q_1}}(x)\ne 0$, noting that $\Tr_{\mathbb{F}_q/\mathbb{F}_{q_1}}(x)\in\mathbb{F}_{q_1}\cap\mathbb{F}_{q_0}$. For any $y\in\mathbb{F}_{q_1}\cap\mathbb{F}_{q_0}$, we have $y(\Tr_{\mathbb{F}_q/\mathbb{F}_{q_1}}(x))^{-1}x\in\mathbb{F}_{q_0}$ and
		\begin{equation*}
		\Tr_{\mathbb{F}_q/\mathbb{F}_{q_1}}(y(\Tr_{\mathbb{F}_q/\mathbb{F}_{q_1}}(x))^{-1}x) = y.
		\end{equation*}
		This shows that $\mathbb{F}_{q_1}\cap\mathbb{F}_{q_0}\subseteq\Tr_{\mathbb{F}_q/\mathbb{F}_{q_1}}(\mathbb{F}_{q_0})$, which completes the proof as $\Tr_{\mathbb{F}_q/\mathbb{F}_{q_1}}(\mathbb{F}_{q_0})\subseteq\mathbb{F}_{q_1}\cap\mathbb{F}_{q_0}$ is obvious.
	\end{proof}
	
	Now let $i$ be an integer such that $|\phi^i|=s$.
	
	\begin{lem}
		\label{l:pre_AGammaL_conju}
		Let $a,b\in\mathbb{F}_q$. Then the following properties hold.
		\begin{itemize}\addtolength{\itemsep}{0.2\baselineskip}
			\item[{\rm (i)}] $\Tr_{\mathbb{F}_q/\mathbb{F}_{q_1}}(a)=\Tr_{\mathbb{F}_q/\mathbb{F}_{q_1}}(b),$ then $(a,1)\phi^i$ and $(b,1)\phi^i$ are $\mathbb{F}_q^+$-conjugate.
			\item[{\rm (ii)}] If $\Tr_{\mathbb{F}_q/\mathbb{F}_{q_1}}(a)\Tr_{\mathbb{F}_q/\mathbb{F}_{q_1}}(b)\neq 0$ , then $(a,1)\phi^i$ and $(b,1)\phi^i$ are $\mathbb{F}_q^+{:}\mathbb{F}_{q_1}^\times$-conjugate.
			\item[{\rm (iii)}] If $\Tr_{\mathbb{F}_q/\mathbb{F}_{q_1}}(a) = 0$ and $\Tr_{\mathbb{F}_q/\mathbb{F}_{q_1}}(b)\ne 0$, then $|(a,1)\phi^i|\ne|(b,1)\phi^i|$.
		\end{itemize}
	\end{lem}
	
	\begin{proof}
		First assume $\Tr_{\mathbb{F}_q/\mathbb{F}_{q_1}}(a)= \Tr_{\mathbb{F}_q/\mathbb{F}_{q_1}}(b)$, so $\Tr_{\mathbb{F}_q/\mathbb{F}_{q_1}}(b-a)=0$.
		Note that $\la \phi^i \ra=\mathrm{Gal}(\mathbb{F}_q/\mathbb{F}_{q_1}).$
		Thus, by the additive form of Hilbert 90, there exists $c\in\mathbb{F}_q$ such that $b- a = c^{\phi^{-i}}-c$.
		This implies that
		\begin{equation*}
		(( a,1)\phi^i)^{(c,1)}
		= (-c,1)( a,1)\phi^i(c,1) = (-c+ a+c^{\phi^{-i}},1)\phi^i = (b,1)\phi^i\\,
		\end{equation*}
		which gives (i).

		Next, assume both $\Tr_{\mathbb{F}_q/\mathbb{F}_{q_1}}(a)$ and $\Tr_{\mathbb{F}_q/\mathbb{F}_{q_1}}(b)$ are in $\mathbb{F}_{q_1}^{\times}$.
		Then there exists $\lambda\in \mathbb{F}_{q_1}^{\times}$ such that $\lambda\Tr_{\mathbb{F}_q/\mathbb{F}_{q_1}}(a)=\Tr_{\mathbb{F}_q/\mathbb{F}_{q_1}}(b)$, and hence $\Tr_{\mathbb{F}_q/\mathbb{F}_{q_1}}(\lambda a)=\Tr_{\mathbb{F}_q/\mathbb{F}_{q_1}}(b)$.
		In view of (i), we see that $(\lambda a, 1)\phi^i$ and $(b,1)\phi^i$ are conjugate in $\mathbb{F}_q^+$.
		Since $\lambda^{\phi^i}=\lambda$, we have
		$((a,1)\phi^i)^{(0,\lambda)}=(a,1)^{(0,\lambda)}\phi^i=(\lambda a, 1)\phi^i$, which implies (ii).
		
		Now assume $\Tr_{\mathbb{F}_q/\mathbb{F}_{q_1}}(a) = 0$ and $\Tr_{\mathbb{F}_q/\mathbb{F}_{q_1}}(b)\ne 0$.  Then
		\begin{equation*}
		((a,1)\phi^i)^s = (a+a^{\phi^{-i}}+\cdots+a^{\phi^{-(s-1)i}},1) = (\Tr_{\mathbb{F}_q/\mathbb{F}_{q_1}}(a),1) = (0,1)
		\end{equation*}
		is the identity of $\Gamma$, whereas
		\begin{equation*}
		((b,1)\phi^i)^s = (\Tr_{\mathbb{F}_q/\mathbb{F}_{q_1}}(b),1) \ne (0,1).
		\end{equation*}
		This shows (iii).
	\end{proof}

	\begin{lem}
		\label{l:pre_AGammaL_conju_subfield}
		Suppose $r\nmid s$ or $r\ne p$. Then for any $a\in\mathbb{F}_q$, there exists $b\in\mathbb{F}_{q_0}$ such that $(a,1)\phi^i$ is conjugate to $(b,1)\phi^i$ by an element in $\mathbb{F}_q^+{:}\mathbb{F}_q^\times$.
	\end{lem}
	
	\begin{proof}
		If $\mathrm{Tr}_{\mathbb{F}_q/\mathbb{F}_{q_1}}(a)=0,$ then Lemma \ref{l:pre_AGammaL_conju}(i) implies $(a,1)\phi^i$ is $\mathbb{F}_q^+$-conjugate to $(0,1)\phi^i.$
		If $\mathrm{Tr}_{\mathbb{F}_q/\mathbb{F}_{q_1}}(a)\neq 0,$ then by Lemma \ref{l:pre_AGammaL_trace}(ii), there exists $b\in \mathbb{F}_{q_0}$ such that $\mathrm{Tr}_{\mathbb{F}_q/\mathbb{F}_{q_1}}(b)\neq 0,$ and hence $(a,1)\phi^i$ is $\mathbb{F}_q^+{:}\mathbb{F}_q^\times$-conjugate to $(b,1)\phi^i$ by Lemma \ref{l:pre_AGammaL_conju}(ii).
	\end{proof}

	\begin{lem}
		\label{l:pre_AGammaL_conju_sub}
		Let $X\leqs\Gamma$ be a subgroup such that $\mathbb{F}_q^+{:}\la\phi\ra_p\leqs X$. Suppose $x\in X$ and $\la x\ra\cap \mathbb{F}_q^+ = 1$. Then $x$ is $X$-conjugate to an element in $X\cap (\mathbb{F}_q^\times{:}\la\phi\ra)$.
	\end{lem}
	
	\begin{proof}
		First note that $\mathbb{F}_q^+{:}\la\phi\ra_p$ is a Sylow $p$-subgroup of $\Gamma$. Thus, conjugating by a suitable element in $X$, we may assume $\la x\ra_p\leqs \mathbb{F}_q^+{:}\la\phi\ra_p$. Let $y$ be a generator of $\la x\ra_p$. Then $y = (a,\phi^i)$ for some $a\in \mathbb{F}_q$ and some integer $i$ such that $\phi^i\in\la\phi\ra_p$. Let $|\phi^i| = s$ and $q = q_1^s$. Then we have
		\begin{equation*}
		y^s=(a+a^{\phi^{-i}}+\dots+a^{\phi^{-(s-1)i}},1)\phi^i=(\mathrm{Tr}_{\mathbb{F}_q/\mathbb{F}_{q_1}}(a),1)\in \mathbb{F}_q^{+}\cap \la y \ra=\{(0,1)\}.
		\end{equation*}
		Thus, by Lemma \ref{l:pre_AGammaL_conju}(i), there exists $b\in\mathbb{F}_q^+$ such that $y^{(b,1)} = \phi^i$, and hence
		\begin{equation*}
		x^{(b,1)}\in C_\Gamma(y^{(b,1)}) = C_\Gamma(\phi^i) = (\mathbb{F}_{q_1}^+{:}\mathbb{F}_{q_1}^\times){:}\la\phi\ra.
		\end{equation*}
		Note that $\mathbb{F}_{q_1}^+\leqslant X$, so $C_X(\phi^i)=\mathbb{F}_{q_1}^+{:}(X\cap(\mathbb{F}_{q_1}^\times{:}\la\phi\ra))$, and $X\cap (\mathbb{F}_{q_1}^{\times}{:}\la \phi \ra_{p'})$ is a Hall-$p'$ subgroup of $C_X(\phi^i)$.
		Thus,
		$\la x\ra_{p'}^{(b,1)z}\leqs X\cap (\mathbb{F}_{q_1}^\times{:}\la\phi\ra_{p'})$ for some $z\in C_X(\phi^i)$. Therefore,
		\begin{equation*}
		\la x^{(b,1)z}\ra = \la x\ra_{p'}^{(b,1)z}\la y\ra^{(b,1)z} =  \la x\ra_{p'}^{(b,1)z}\la\phi^i\ra \leqs \mathbb{F}_{q_1}^\times{:}\la\phi\ra
		\end{equation*}
		and the proof is complete by noting that $(b,1)z\in X$.
	\end{proof}
	\section{Two-dimensional linear groups}
	
	\label{s:PSL2}
	
	Let $G\leqs\mathrm{Sym}(\Omega)$ be a primitive group with socle $G_0 = \PSL_2(q)$ for some prime power $q = p^f$, and let $H$ be a point stabiliser of $G$. Write $G = G_0.O$ for some $O\leqs\Out(G_0)$. For any subgroup $K$ of $G$, we write $K_0:=K\cap G_0$ (in particular, $H_0:=H\cap G_0$). In this section, we will study the large fixers of $G$.
	
	For convenience, we also set up the notation of some subgroups of $G$, with respect to a fixed basis of $\mathbb{F}_q^2$. Let $Q,D\leqs \PGL_2(q)$ be the images of the subgroups of $\GL_2(q)$ consisting of all the upper-triangular unipotent matrices and all the diagonal matrices, respectively. That is,
	\begin{equation}
	\label{e:Q}
	Q:=\left\{xZ(\GL_2(q)):x =
	\begin{pmatrix}
	1&a\\
	&1
	\end{pmatrix}
	\mbox{ for some }a\in\mathbb{F}_q
	\right\}
	\end{equation}
	and
	\begin{equation}
	\label{e:D}
	D:=\left\{xZ(\GL_2(q)):x =
	\begin{pmatrix}
	\lambda&\\
	&\mu
	\end{pmatrix}
	\mbox{ for some }\lambda,\mu\in\mathbb{F}_q^*
	\right\}.
	\end{equation}
	Note that $Q\cong C_p^f$ and $D\cong C_{q-1}$. Moreover, we have
	\begin{equation*}
	N_{G}(Q) = (Q{:}(D{:}\la \phi\ra))\cap G = Q{:}(G\cap (D{:}\la \phi\ra))
	\end{equation*}
	is a maximal subgroup of $G$ of type $P_1$. In particular,
	\begin{equation*}
	N_{\PGammaL_2(q)}(Q) = (Q{:}D){:}\la\phi\ra\cong\mathrm{A\Gamma L}_1(q),
	\end{equation*}
	so we define the following isomorphism $\rho:N_{\PGammaL_2(q)}(Q)\to \Gamma$ as follows (recall that $\Gamma$ is defined in \eqref{e:Gamma}):
	\begin{equation}
	\label{e:pi}
	\rho:
	\begin{pmatrix}
	1&a\\
	&1
	\end{pmatrix}
	\begin{pmatrix}
	\lambda&\\
	&\mu
	\end{pmatrix}
	Z(\GL_2(q))\phi^i\mapsto (a,\lambda\mu^{-1})\phi^i.
	\end{equation}
	
	We first establish Theorem \ref{thm:PSL2} and Corollary \ref{cor:rho_0_PSL2} when $q$ is small.
	
	\begin{prop}
		\label{p:PSL2_61}
		The statements of Theorem \ref{thm:PSL2} and Corollary \ref{cor:rho_0_PSL2} hold for $q\leqs 61$.
	\end{prop}
	
	\begin{proof}
		This can be done easily with the aid of {\sc Magma} \cite{Magma}.
	\end{proof}
	
	Thus, we assume $q > 61$ from now on, and we list the maximal subgroups of $G$ up to $G$-conjugacy in Table \ref{tab:max_PSL2} (in fact, the description of the maximal subgroups in the table holds for all $q > 11$).
	
	{\small
		\begin{table}
			\[
			\begin{array}{lllll} \hline
			\mbox{Type}&M_0&O&\mbox{Conditions}&\mbox{Conjugates}\\\hline
			P_1 & C_p^f{:}C_{(q-1)/(2,q-1)} & & &\\
			\GL_1(q)\wr S_2 & D_{2(q-1)/(2,q-1)} & & &\\
			\GL_1(q^2) & D_{2(q+1)/(2,q+1)} & & &\\
			\GL_2(q_0) & \PGL_2(q_0) & \leqs\la\phi\ra & \mbox{$q = q_0^2$ odd} &\delta\\
			& \PSL_2(q_0) & & \mbox{$q = q_0^r$ odd, $r$ odd prime} &\\
			& & & \mbox{$q = q_0^r$ even, $q_0\ne 2$, $r$ prime} &\\
			2^{1+2}_-.\mathrm{O}_2^-(2) & S_4 & =1 & \mbox{$q = p\equiv \pm 1\pmod 8$} & \delta \\
			& A_4 & \leqs\la\delta\ra & \mbox{$q = p\equiv \pm 3,5,\pm 13\pmod{40}$} &  \\
			& & =\la\delta\ra & \mbox{$q = p\equiv \pm 11, \pm 19\pmod{40}$} &  \\
			A_5 & A_5 & =1 & \mbox{$q = p\equiv \pm 1\pmod{10}$} & \delta \\
			& & \leqs\la \phi\ra & \mbox{$q = p^2$, $p\equiv \pm 3\pmod{10}$} & \delta \\
			\hline
			\end{array}
			\]
			\caption{Maximal subgroup $M$ of $G$ with $\soc(G) = \PSL_2(q)$ for $q > 61$}
			\label{tab:max_PSL2}
	\end{table}}
	
	\begin{rem}
		\label{r:tab:max_PSL2}
		Let us give some comments on Table \ref{tab:max_PSL2}.
		\begin{itemize}\addtolength{\itemsep}{0.2\baselineskip}
			\item[{\rm (i)}] The definition of types of subgroups is followed from Kleidman and Liebeck \cite{KL_classical}, which roughly describes the structure of $M$.
			\item[{\rm (ii)}] In the third column, we record the condition of $O\leqs\Out(G_0)$ such that $M_0.O$ is a maximal subgroup of $G_0.O$. We refer the reader to \cite[Tables 8.1 and 8.2]{BHR_classical} for necessary and sufficient conditions of the existence and maximality of some subgroups.
			\item[{\rm (iii)}] The fifth column describes the number of conjugates of the relevant maximal subgroups in $G$. More precisely, if ``$\delta$" appears in that column, then there are exactly two $G$-conjugacy classes of such maximal subgroups, namely $M$ and $M^\delta$, otherwise there is a unique $G$-conjugacy classes of such maximal subgroups.
		\end{itemize}
	\end{rem}
	
	The following is \cite[Satz 8.27]{H_book}.
	
	\begin{lem}
		\label{l:PSL2_sub}
		Any subgroup of $G_0$ is isomorphic to one of the following groups:
		\begin{itemize}\addtolength{\itemsep}{0.2\baselineskip}
			\item[{\rm (i)}] $C_p^m{:}C_t$ for some $m\leqs f$ with $t\mid (p^m-1,p^f-1)$;
			\item[{\rm (ii)}] $C_m$ or $D_{2m}$ with $m\mid (p^f\pm 1)/(2,q-1)$;
			\item[{\rm (iii)}] $D_8$, $A_4$, $S_4$ or $A_5$;
			\item[{\rm (iv)}] $\PSL_2(p^m)$ with $m\mid f$, or $\PGL_2(p^m)$ with $2m\mid f$.
		\end{itemize}
	\end{lem}
	
	\begin{cor}
		\label{c:PSL2_p-element}
		Let $K$ be a subgroup of $G$ which does not contain a unipotent element of prime order. Then $K_0$ is isomorphic to $A_4$, $S_4$, $A_5$, or a cyclic group, or a dihedral group.
	\end{cor}
	
	The main theorem of \cite{CJ_cyclic} classifies the finite simple groups whose cyclic subgroups of same order are conjugate. In particular, we have the following.
	
	\begin{lem}
		\label{l:PSL2_cyclic_conju}
		If $q$ is even or $f$ is odd, then any two cyclic subgroups of $G_0$ of the same order are $G_0$-conjugate.
	\end{lem}
	
	The following lemma is \cite[Satz 8.5]{H_book}.
	
	\begin{lem}
		\label{l:PSL2_ele_cyc}
		Let $x\in G_0$. Then either $|x| = p$, or $x$ is contained in a cyclic subgroup of $G_0$ of order $(q-1)/(2,q-1)$ or $(q+1)/(2,q-1)$.
	\end{lem}

	\begin{lem}
		\label{l:PSL2_ele_cyc_PGL}
		Let $m\in\Spec(\PGL_2(q))$. Then either $m = p$, or $m$ divides $q-1$ or $q+1$.
	\end{lem}

	Now we consider each primitive action of $G$ in turn.
	
	\begin{lem}
		\label{l:PSL2_para}
		Suppose $H$ is of type $P_1$. Then $G$ has no large fixer.
	\end{lem}
	
	\begin{proof}
		Note that $G$ is a $2$-transitive permutation group, so $G$ has the EKR property by Theorem \ref{t:2-trans}. This implies that $|K|\leqs |H|$ for any fixer $K$ of $G$. Hence, we only need to consider the case where $|K| = |H|$. By inspecting Table \ref{tab:max_PSL2}, we deduce that $K$ is conjugate to $H$ in $G$, which completes the proof.
	\end{proof}

	\begin{lem}
		\label{l:PSL2_C6_S}
		Suppose $q > 61$ and $H$ is of type $2^{1+2}_-.\mathrm{O}_2^-(2)$ or $A_5$, and let $K$ be a large fixer of $G$. Then $K = H^\delta$ and one of the following holds:
		\begin{itemize}\addtolength{\itemsep}{0.2\baselineskip}
			\item[{\rm (i)}] $G = \PSL_2(p)$, $p\equiv \pm 1\pmod 8$ and $H \cong S_4$;
			\item[{\rm (ii)}] $G = \PSL_2(p)$, $p\equiv \pm 1\pmod{10}$ and $H \cong A_5$;
			\item[{\rm (iii)}] $G = \PSL_2(p^2)$, $3\ne p\equiv\pm 3\pmod{10}$ and $H\cong A_5$.
		\end{itemize}
	\end{lem}
	
	\begin{proof}
		First assume $H$ is of type $2^{1+2}_-.\mathrm{O}_2^-(2)$, so $q = p$ is a prime, and $H$ does not contain a unipotent element. Note that $\Spec(H) \subseteq \{1,2,3,4\}$, so if $K_0$ is cyclic or dihedral, then $|K_0| < |H_0|$ and so $|K| < |H|$ by Lemma \ref{l:pre_as_size}. In view of Corollary \ref{c:PSL2_p-element}, it suffices to consider the cases where $K_0\cong H_0$, so $G = \PSL_2(p)$, $p\equiv \pm 1\pmod 8$ and $H\cong S_4$ by inspecting Table \ref{tab:max_PSL2}. By Lemma \ref{l:PSL2_cyclic_conju}, we see that $K = H^\delta$ is indeed a fixer of $G$, and this case is recorded as in (i).
		
		Now assume $H$ is of type $A_5$, so $H_0\cong A_5$. By arguing as above, we see that the only possibility is $K = H^\delta$. If $G = G_0$, then by Lemma \ref{l:PSL2_ele_cyc}, any cyclic subgroup of $K$ is conjugate to a subgroup of $H$, noting that there is no element of order $p$ in $K$.
		Thus, if $G = G_0$ then $K = H^\delta$ is a fixer of $G$, as recorded in (ii) and (iii). Finally, assume $G = \PSigmaL_2(q)$, so $H\cong K\cong S_5$. Note that $x$ and $x^\delta$ are not $G$-conjugate if $x\in K\setminus K_0$ is an involution (see \cite[Proposition 3.2.9]{BG_classical}). Moreover, any involution in $H\setminus H_0$ are $H$-conjugate to $x^\delta\in H$. It follows that $x^G\cap H = \emptyset$, and so $K$ is not a fixer of $G$.
	\end{proof}
	
	%
	%
	%
	%

	\begin{lem}
		\label{l:PSL2_C3}
		Suppose $q> 61$ and $H$ is of type $\GL_1(q^2)$. Then $G$ has no large fixer.
	\end{lem}
	
	\begin{proof}
		Suppose $K$ is a large fixer of $G$, and let $\widehat{K}$ be a maximal overgroup of $K$ in $G$. Since $|K|\geqs |H|$, we have $\widehat{K}$ is of type $P_1$ or $\GL_2(q^{1/2})$ by inspecting Table \ref{tab:max_PSL2}.
		
		First assume $\widehat{K}$ is of type $P_1$, so $\widehat{K}_0 = (C_p^f{:}C_{(q-1)/(2,q-1)})$. Note that
		\begin{equation}
		\label{e:gcd}
		\left(\frac{q-1}{(2,q-1)},\frac{q+1}{(2,q-1)}\right) = 1.
		\end{equation}
		As $K_0$ contains an element of order $|K_0|_{p'}$, it follows that $|K_0|_{p'}\in\Spec(H_0)$, and so  $|K_0|_{p'}\leqs 2$ by \eqref{e:gcd}. Note that if $p$ is odd, then $H_0$ does not contain an element of order $p$, which yields $|K_0|_p = 1$ and thus $|K_0|\leqs 2 < |H_0|$. If $q$ is even, then
		\begin{equation*}
		|K_0| \leqs  2|K_0|_p\leqs 2q < 2(q+1) = |H_0|.
		\end{equation*}
		We eliminate both cases by Lemma \ref{l:pre_as_size}.
		
		Thus, to complete the proof, we may assume $\widehat{K}$ is of type $\GL_2(q^{1/2})$, so $\widehat{K}_0\cong\PGL_2(q^{1/2})$. Observe that
		\begin{equation*}
		\left(q^{1/2}\pm 1,\frac{q+1}{(2,q-1)}\right)\text{ divides }\left(q-1,\frac{q+1}{(2,q-1)}\right)\text{, which divides }(2,q-1).
		\end{equation*}
		This shows that $|K_0|\leqs (2,q-1)|\widehat{K}_0|_p\leqs (2,q-1)q^{1/2} < |H_0|$. Again, this is incompatible with Lemma \ref{l:pre_as_size}, so we conclude the proof.
	\end{proof}
	
	Now we turn to two primitive actions of $G$ that need special attention. As above, let $\widehat{K}$ be a maximal overgroup of $K$ in $G$.
	
	\begin{lem}
		\label{l:PSL2_C2}
		Suppose $H$ is of type $\GL_1(q)\wr S_2$ and $q > 61$. If $K$ is a large fixer of $G$, then $q$ is even, and $\widehat{K}$ is of type $P_1$ or $\GL_2(q^{1/2})$. Moreover, if $q$ is even, $G = G_0$ and $K$ is a maximal subgroup of type $P_1$ or $\GL_2(q^{1/2})$, then $K$ is a large fixer of $G$.
	\end{lem}
	
	\begin{proof}
		First assume $K$ is a large fixer of $G$, noting that $|\widehat{K}|\geqs |K|\geqs |H|$. If $\widehat{K}$ is of type $\GL_1(q^2)$, then $\widehat{K}$ is clearly not a fixer of $G$, and any proper subgroup of $\widehat{K}$ has order strictly less than $|H|$. Thus, by inspecting Table \ref{tab:max_PSL2}, this implies that $\widehat{K}$ is of type $P_1$ or $\GL_2(q^{1/2})$.
		
		Now assume $q$ is odd, noting that $H_0 \cong D_{q-1}$ does not contain an element of order $p$. This implies that a large fixer $K$ does not contain an element of order $p$. If $\widehat{K}$ is of type $P_1$, then $x$ lies in a cyclic subgroup of $G_0$ of order $(q-1)/2$, which implies that $|K_0| \leqs (q-1)/2 < |H_0|$, a contradiction to Lemma \ref{l:pre_as_size}. If $\widehat{K}$ is of type $\GL_2(q^{1/2})$, then $\widehat{K}_0\cong \PGL_2(q^{1/2})$, and by Corollary \ref{c:PSL2_p-element} we see that $|K_0| \leqs \max(|D_{2(q^{1/2}+1)}|, |A_5|) < |H_0|$, which is incompatible with Lemma \ref{l:pre_as_size} once again.
		
		To complete the proof, we assume $q$ is even and $G = G_0$, so $H\cong D_{2(q-1)}$. First assume $K$ is a maximal subgroup of type $P_1$ and let $x\in K$. Note that either $|x| = 2$ or $x$ lies in a cyclic subgroup of $K$ of order $2^f-1$. By Lemma \ref{l:PSL2_cyclic_conju}, we see that $x$ is $G$-conjugate to an element in $H$, so $K$ is indeed a fixer of $G$. Finally, let us assume $K$ is a maximal subgroup of type $\GL_2(q^{1/2})$, so $K\cong\PSL_2(q^{1/2})$. Let $x\in K$, noting by Lemma \ref{l:PSL2_ele_cyc} that either $x$ is contained in a cyclic subgroup of $K$ of order $q^{1/2}\pm 1$, or $|x| = 2$. In the former case, $x$ is contained in a cyclic subgroup of $G$ of order $q-1$, so $x$ is conjugate to an element in $H$ by Lemma \ref{l:PSL2_cyclic_conju}, while in the latter case it is also easy to see that $x^G\cap H\ne \emptyset$. It follows that $K$ is a fixer of $G$.
	\end{proof}
	
	Finally, let us consider the case where $H$ is a subfield subgroup.
	
	\begin{lem}
		\label{l:PSL2_C5_even}
		Suppose $H$ is of type $\GL_2(q^{1/2})$. If $K$ is a large fixer of $G$, then $q$ is even and $\widehat{K}$ is of type $P_1$. Moreover, if $q$ is even, $G = G_0$ and $K \cong C_2^f{:}C_{2^{f/2}+1}$, then $K$ is a large fixer of $G$.
	\end{lem}
	
	\begin{proof}
		First assume $K$ is a large fixer of $G$. Then by inspecting Table \ref{tab:max_PSL2}, either $\widehat{K}$ is of type $P_1$, or $q$ is odd, $G\leqs\PSigmaL_2(q)$ and $K = H^\delta$ (note that $|\widehat{K}|\geqs |H|$). Let $x\in H_0$ is an element of order $p$. As noted in \cite[Remark 5.5.2(ii)]{BG_classical}, any element of $H_0$ of order $p$ is $G_0$-conjugate to $x$. Thus, if $q$ is odd and $K = H^\delta$, then $x^\delta\in K_0$ is not $G_0$-conjugate to any element in $H_0$ (see \cite[Lemma 3.2.8]{BG_classical}), so $K$ is not a fixer of $G$ by Lemma \ref{l:pre_red}. Thus, we may assume $\widehat{K}$ is of type $P_1$. Suppose $q$ is odd. Note that $\widehat{K}_0$ contains a Sylow $p$-subgroup of $G_0$, so
		\begin{equation*}
		|x^{G_0}\cap \widehat{K}_0| = |(x^\delta)^{G_0}\cap \widehat{K}_0| = (q-1)/2,
		\end{equation*}
		and thus $|K_0|_p\leqs (q+1)/2$. Moreover, as a Hall $p'$-subgroup of $\widehat{K}_0$ is cyclic, we see that the order of a Hall $p'$-subgroup of $K_0$ lies in $\Spec(H_0) = \Spec(\PGL_2(q^{1/2}))$. By Lemma \ref{l:PSL2_ele_cyc_PGL}, it follows that $|K_0|_{p'}\leqs q^{1/2}+1$. This implies that
		\begin{equation*}
		|K_0| = |K_0|_p\cdot |K_0|_{p'} \leqs (q^{1/2}+1)(q+1)/2 < |\PGL_2(q^{1/2})| = |H_0|,
		\end{equation*}
		which is incompatible with Lemma \ref{l:pre_as_size}. This shows that $q$ is even and $\widehat{K}$ is of type $P_1$.
		
		Finally, if $q$ is even, $G = G_0$ and $K \cong C_2^f{:}C_{2^{f/2}+1}$ is a subgroup of a maximal subgroup of type $P_1$, then we see that $K$ is a fixer of $G$ by applying Lemma \ref{l:PSL2_cyclic_conju}.
	\end{proof}
	
	\begin{lem}
		\label{l:PSL2_C5_odd}
		Suppose $q > 61$ and $H$ is of type $\GL_2(q^{1/r})$ with $r\geqs 3$ a prime. If $K$ is a large fixer of $G$, then there exists a maximal overgroup $\widehat{K}$ of $K$ of type $P_1$. Moreover, if $G = G_0$ and $K \cong C_p^f{:}C_{(q^{1/r}-1)/(2,q-1)}$, then $K$ is a large fixer of $G$.
	\end{lem}
	
	\begin{proof}
		First assume $K$ is a large fixer of $G$, noting that
		\begin{equation}
		\label{e:gcd_1/r}
		\left(\frac{q\pm 1}{(2,q-1)},\frac{q^{1/r}\mp 1}{(2,q-1)}\right)\text{ divides }\left(\frac{q\pm 1}{(2,q-1)},\frac{q\mp 1}{(2,q-1)}\right) = 1.
		\end{equation}
		Thus, if $\widehat{K}$ is of type $\GL_1(q)\wr S_2$ or $\GL_1(q^2)$, then $\widehat{K}_0 \cong D_{2(q\pm 1)/(2,q-1)}$ and so $K_0\leqs D_{2(q^{1/r}\pm 1)/(2,q-1)}$ by applying \eqref{e:gcd_1/r}, which implies that $|K_0|\leqs 2(q^{1/r}+1) < |H_0|$, a contradiction to Lemma \ref{l:pre_as_size}.
		
		Now assume $\widehat{K}$ is of type $\GL_2(q^{1/s})$ for some prime $s\mid f$, and in view of Lemma \ref{l:pre_red} we may assume $G = G_0$. Then $s\leqs r$ as $|\widehat{K}|\geqs|K|\geqs|H|$. If $s = r$ then $\widehat{K}$ is $G$-conjugate to $H$, and so $K$ is a stable fixer, a contradiction. Thus, we may assume $s < r$, and let $\widetilde{K}$ be the maximal overgroup of $K$ in $\widehat{K}$. First note that if $\widetilde{K}$ is of type $P_1$ (as a subgroup of $\widehat{K}$) then $K$ has a maximal overgroup in $G$ of type $P_1$, so we do not need to consider this case. Similarly, if $\widetilde{K}$ is of type $\GL_1(q^{1/s})\wr S_2$ or $\GL_1(q^{2/s})$, then $K$ is contained in a maximal subgroup of $G$ of type $\GL_1(q)\wr S_2$ or $\GL_1(q^2)$, and we note that both cases have been handled above. If $\widetilde{K}$ is of type $2_-^{1+2}.\mathrm{O}_2^-(2)$ or $A_5$, then $q^{1/s} = p$ or $p^2$ by inspecting Table \ref{tab:max_PSL2}, so $s$ is the largest prime divisor of $f$, a contradiction to $s<r$. It suffices to consider the case where $\widetilde{K}$ is of type $\GL_2(q^{1/t})$ for some $t$ with $s\mid t$. This can be done by repeating the above argument.
		
		This shows that $K$ has a maximal overgroup $\widehat{K}$ in $G$ of type $P_1$. Finally, let $G = G_0$ and $K\cong C_p^f{:}C_{(q^{1/r}-1)/(2,q-1)}$. Then by combining Lemma \ref{l:pre_red} and the argument in Example \ref{ex:subfield}, we see that $K$ is a fixer of $G$. 
	\end{proof}

	To summarise, we have the following proposition.
	
	\begin{prop}
		\label{p:PSL2_summarise}
		Let $K$ be a large fixer of $G$ and assume $q>61$. Then one of the following holds:
		\begin{itemize}\addtolength{\itemsep}{0.2\baselineskip}
			\item[{\rm (a)}] $q$ is even, $H$ is of type $\GL_1(q)\wr S_2$ and $\widehat{K}$ is of type $P_1$;
			\item[{\rm (b)}] $q$ is even, $H$ is of type $\GL_1(q)\wr S_2$ and $\widehat{K}$ is of type $\GL_2(q^{1/2})$;
			\item[{\rm (c)}] $q$ is even, $H$ is of type $\GL_2(q^{1/2})$ and $\widehat{K}$ is of type $P_1$;
			\item[{\rm (d)}] $H$ is of type $\GL_2(q^{1/r})$ with $r$ odd, and $\widehat{K}$ is of type $P_1$.
		\end{itemize}
	\end{prop}
	
	Each case appearing in Proposition \ref{p:PSL2_summarise} requires more refined treatment, and we will work on it in Section \ref{s:thm1}. To do this, the following lemma will be useful. Here we note that $G_0\cong \SL_2(q)$ if $q$ is even. Let
	\begin{equation}
	\label{e:SL_invo}
	z =
	\begin{pmatrix}
	0&1\\
	1&0
	\end{pmatrix}\in \SL_2(q)
	\end{equation}
	and suppose $G = \la G_0,\psi\ra$, where $\psi\in\la\phi\ra$.
	
	\begin{lem}
		\label{l:PSL2_even_conju}
		Suppose $q$ is even. Let $x\in G$ be of even order, and suppose the order of $G_0x$ in $G/G_0$ is odd. Then $\la x^2\ra$ is $G$-conjugate to a subgroup of $\la\psi\ra_{2'}$ and $\la x\ra$ is $G$-conjugate to a subgroup of $\la z\ra\times \la\psi\ra_{2'}$.
	\end{lem}
	
	\begin{proof}
		Let $t = |G_0x|$. Then $|x| = 2t$ and $x^t$ is an involution in $G_0$. Note that any element in $G_0$ of even order is an involution, and all involutions in $G_0$ are conjugate. Thus, there exists $g\in G_0$ such that $(x^t)^g = z$ and so $x^g\in C_G(z) = P{:}\la\psi\ra$, where $P$ is a Sylow $2$-subgroup of $G_0$ containing $z$. In particular, $(x^2)^g\in C_G(z)$ is an element of odd order $t$. It follows that there exists $g_1\in C_G(z)$ such that $(x^2)^{gg_1}$ is contained in $\la\psi\ra_{2'}$ since $\la\psi\ra_{2'}$ is a Hall $2'$-subgroup of the soluble group $C_G(z)$. Therefore,
		\begin{equation*}
		\la x\ra^{gg_1} = \la x^t\ra^{gg_1} \times \la x^2\ra^{gg_1} \leqs \la z\ra \times \la\psi\ra_{2'},
		\end{equation*}
		which concludes the proof.
	\end{proof}

	\section{Proof of Theorem \ref{thm:PSL2}}
	
	\label{s:thm1}
	
	In this section, we will establish Theorem \ref{thm:PSL2}, and we adopt the notation in Section \ref{s:PSL2}. In view of Proposition \ref{p:PSL2_61}, we may also assume $q > 61$, so it suffices to consider the four cases recorded in Proposition \ref{p:PSL2_summarise}.

	\subsection{Case (a)}
	
	First let us consider case (a) of Proposition \ref{p:PSL2_summarise}.
	
	\begin{lem}
		\label{l:PSL2_(a)}
		Suppose $|K|\geqs |H|$ and assume Proposition \ref{p:PSL2_summarise}(a). Then $K$ is a fixer of $G$ if and only if $K\leqs\widehat{K}_0.\la\psi\ra_{2'}$.
	\end{lem}
	
	\begin{proof}
		First assume $K$ is a fixer of $G$ and suppose that $|K:K_0|$ is even. Then there exists $x\in K\setminus K_0$ such that $|G_0x| = 2$ and $|x|$ is a $2$-power, so $\la G_0,x\ra = G_0.\la \varphi\ra$, where $\varphi$ is an involutory field automorphism. It follows that $x$ is $G$-conjugate to an element of $H_0.\la\varphi\ra\cong C_{q-1}{:}C_2^2$, which contains no element of order $4$. Thus, $x$ is an involution, so by \cite[Proposition 4.9.1(d)]{GLS_CFSG3}, $x$ is $G_0$-conjugate to $\varphi$. With this in mind, we may assume $\varphi\in K$ by considering a suitable $G$-conjugate of $K$. By arguing as above, $y\varphi$ is an involution for any involution $y\in K_0$. This implies that $C_{G_0}(\varphi)\cong \SL_2(q^{1/2})$ contains a Sylow $2$-subgroup $P$ of $K_0$, which yields $|P|\leqs q^{1/2}$ and hence $|K_0|_p\leqs q^{1/2}$. It follows from Lemma \ref{l:PSL2_sub}(i) that $|K_0|_{p'}\leqs q^{1/2}-1$. Therefore, $|K_0|\leqs q^{1/2}(q^{1/2}-1) < 2(q-1) = |H_0|$, which is incompatible with Lemma \ref{l:pre_as_size}.
		
		It suffices to show that $K = \widehat{K}_0.\la\psi\ra_{2'}$ is a fixer of $G$. To see this, we may assume $K = (Q{:}D){:}\la\psi\ra_{2'}$ (recall that $Q$ and $D$ are defined in \eqref{e:Q} and \eqref{e:D}, respectively). With a suitable conjugation of $H$, we may also assume that $H = D{:}(\la z\ra\times\la\psi\ra)$, where $z$ is defined as in \eqref{e:SL_invo}. Let $x\in K$. If $|x|$ is odd, then $x$ is $K$-conjugate to an element in a Hall $2'$-subgroup $D{:}\la\psi\ra_{2'}$ of $K$, so $x^G\cap H\ne \emptyset$. If $|x|$ is even, then we can apply Lemma \ref{l:PSL2_even_conju} (note that $|K:K_0|$ is odd), so $x$ is $G$-conjugate to an element in $\la z\ra\times\la\psi\ra_{2'}\leqs H$. This shows that $K$ is indeed a fixer of $G$, which completes the proof.
	\end{proof}

	\subsection{Case (b)}
	
	Next, we turn to case (b) of Proposition \ref{p:PSL2_summarise}.
	
	\begin{lem}
		\label{l:PSL2_(b)}
		Assume Proposition \ref{p:PSL2_summarise}(b). Then $\widehat{K}$ is a large fixer of $G$.
	\end{lem}
	
	\begin{proof}
		In view of Lemma \ref{l:pre_red}, we may assume $G=\PGammaL_2(q)$,  and we also assume $H=D{:}(\la z \ra \times \la \phi \ra)$, where $z$ is described as in \eqref{e:SL_invo}. For convenience, we write $K=C_G(\phi^{f/2})$, so ${K}_0\cong \SL_2(2^{f/2})$ is the matrix group comprising all $2\times 2$ matrices of determinant $1$ over $\mathbb{F}_{2^{f/2}}$.
		
		Let $x\in K$. We first claim that there is a soluble non-cyclic subgroup of $K$ containing $x$.
		Let $M\leqs K$ be a minimal non-cyclic group containing $x$ (that is, any proper subgroup of $M$ containing $x$ is cyclic).
		To prove the claim, it suffices to show that $M$ is soluble.
		We argue by contradiction and suppose that $M$ is insoluble.
		Then $M_0 = M\cap G_0$ is insoluble, and hence $M_0\cong \SL_2(2^t)$ for some $t$ dividing $f/2$ by Lemma \ref{l:PSL2_sub}. Note that the finite groups with a cyclic maximal subgroup are classified in \cite[Theorem 1]{PK_max_cyclic}. Thus, by inspecting the list of groups described in \cite[Theorem 1]{PK_max_cyclic}, and with the observation $M_0\cong\SL_2(2^t)$ in mind, we see that any maximal subgroup of $M$ is not cyclic. In other words, there is a non-cyclic maximal subgroup of $M$ containing $x$, which is incompatible with the minimality of $M$. This completes the proof of the claim.
		
		Now let $M\leqs K$ be a soluble non-cyclic group containing $x$ (the existence of $M$ follows from the claim above).
		Observe that $M_0 \ne 1$ since $M$ is not cyclic. Once again, by applying Lemma \ref{l:PSL2_sub}, we see that either $M_0\cong C_2^r{:}C_t$ with $t$ divides $(2^r-1,2^{f/2}-1)$, or $M_0\cong C_t$ or $D_{2t}$ with $t\mid 2^{f/2}\pm 1$.
		
		Suppose $M_0$ is cyclic or dihedral. Then $M$ contains a normal cyclic subgroup $C$ of order $t$ with $t\mid 2^{f/2}\pm 1$.
		By Lemma \ref{l:PSL2_cyclic_conju}, $C$ is $K_0$-conjugate to a subgroup of $D$, where $D\leqs G$ is a subgroup defined in \eqref{e:D}, noting that $2^f-1 = (2^{f/2}+1)(2^{f/2}-1)$.
		It follows from Lemma \ref{l:PSL2_ele_cyc} that $M\leqslant N_G(C)$ is conjugate to a subgroup of $H=N_G(D)$.
		
		Therefore, it remains to consider the case where $M_0\cong C_2^r{:}C_t$.
		Here $M$ has a unique Sylow $2$-subgroup. As noted in \cite[Satz 8.2(c)]{H_book}, two distinct Sylow $2$-subgroups of $K_0$ intersect trivially, and thus $M$ normalises a unique Sylow $2$-subgroup of $K_0$.
		With this in mind, we may replace $M$ and $x$ with suitable $K_0$-conjugates, and assume $x\in M\leqs N_K(Q\cap K_0)$ (recall \eqref{e:Q} for the definition of $Q$).
		To prove that $x^G\cap H\neq \emptyset,$ we first note that
		\begin{equation*}
		\rho(N_K(Q\cap K_0))=\{(a,b)\phi^i\mid a\in \mathbb{F}_{2^{f/2}}, b\in\mathbb{F}_{2^{f/2}}^\times, 1\leqs i \leqs |\phi|  \} =  (\mathbb{F}_{2^{f/2}}^+{:}\mathbb{F}_{2^{f/2}}^{\times}){:}\la \phi \ra.
		\end{equation*}
		If $\la x \ra \cap Q=1,$ then $x$ is $N_G(Q)$-conjugate to an element in $D{:}\la \phi \ra\leqs H$ by Lemma \ref{l:pre_AGammaL_conju_sub}.
		Thus, we may assume $|\la x \ra \cap Q|=2$.
		If $|G_0x|$ is odd, then Lemma \ref{l:PSL2_even_conju} implies that $x$ is $G$-conjugate to an element in $\la z \ra \times \la \phi \ra\leqs H,$ and hence we may further assume $|G_0x|$ is even.
		Let $y\in \la x \ra \cap Q$ be a non-identity element.
		Then there exists $g\in N_{K_0}(Q\cap K_0)$ such that $\rho(y^g) = (1,1)$, so
		\begin{equation*}
		\rho(x^g)\in C_{\rho(N_K(Q\cap K_0))}(\rho(y^g)) = C_{\rho(N_K(Q\cap K_0))}((1,1)) =\mathbb{F}_{2^{f/2}}^+{:}\la \phi \ra.
		\end{equation*}
		It follows that $\rho(x^g) = (a,1)\phi^i$ for some $a\in\mathbb{F}_{2^{f/2}}$ and some integer $i$ with $s:=|\phi^i| = |G_0x|$, so $s$ is even by our assumption. Suppose $q_1^s = q$, noting that $\mathbb{F}_{q_1}\subseteq \mathbb{F}_{2^{f/2}}$ as $s$ is even. Then $\Tr_{\mathbb{F}_q/\mathbb{F}_{q_1}}(a) = 0$ by Lemma \ref{l:pre_AGammaL_trace}(i), whence $\rho(x^g)$ is $\mathbb{F}_q^+$-conjugate to $\phi^i$ by Lemma \ref{l:pre_AGammaL_conju}(i). This shows that $x$ is $G$-conjugate to $\phi^i\in H$, and we conclude the proof.
	\end{proof}
	

	\subsection{Case (c)}
	
	Now we consider Proposition \ref{p:PSL2_summarise}(c), and we write $G = G_0.\la\psi\ra$ for some $\psi\in\la\phi\ra$. Without loss of generality, we may assume $\widehat{K} = N_G(Q)$. Let $C\leqs D$ be the unique subgroup of order $2^{f/2}+1$, where $D$ is defined in \eqref{e:D}, and let $L = Q{:}(C{:}\la\psi\ra)$.

	\begin{lem}
		\label{l:PSL2_(c)}
		Assume Proposition \ref{p:PSL2_summarise}(c). Then $K$ is a large fixer of $G$ if and only if $|G:G_0|$ is odd and $K = L$.
	\end{lem}
	
	\begin{proof}
		First assume $G = G_0$. Note that $K$ has an element of order $|K|_{2'}$, so if $K$ is a fixer of $G$ then $|K|_{2'}\leqs q^{1/2}+1$ by Lemma \ref{l:PSL2_ele_cyc}. Now if $|K|_{2'} < q^{1/2}+1$ then $|K|_{2'} \leqs q^{1/2}-1$, which yields $|K| < |H|$. Thus, $K$ is a large fixer of $G$ only if $K = L$, and we note that $L$ is a fixer of $G$ by Lemma \ref{l:PSL2_C5_even}.
		
		It follows by Lemmas \ref{l:pre_red} and \ref{l:pre_as_size} that $K$ is a large fixer of $G$ only if $K_0 = L_0$. Moreover, since $|H:H_0| = |G:G_0|$ and
		\begin{equation*}
		\frac{|K_0|}{|H_0|} = \frac{q^{1/2}}{q^{1/2}-1},
		\end{equation*}
		we see that $|K| \geqs |H|$ if and only if $|K:K_0| = |G:G_0|$.
		
		Assume $|G:G_0| = |\psi|$ is even. Let $\varphi\in\la\phi\ra$ be an involutory field automorphism. Then $\varphi\in G$ and we may assume that $\varphi\in K$ and $H = C_G(\varphi)$. Here we note that $H_0.\la \varphi\ra\cong \SL_2(q^{1/2})\times C_2$, so it does not contain an element of order $4$ since its Sylow $2$-subgroups are elementary abelian. Thus, by arguing as in the proof of Lemma \ref{l:PSL2_(a)}, we see that if $K$ is a fixer of $G$, then for any $x\in K$ with $2$-power order and $|G_0x| = 2$, we have $x$ is an involution. This implies that $y\varphi$ is an involution for any $y\in Q\leqs K_0$, which yields $H = C_G(\varphi)$ contains the Sylow $2$-subgroup $Q$ of $K_0$. This is clearly impossible.
		
		Therefore, $|K:K_0|=|G:G_0|$ is odd, and we claim that $K$ is $G$-conjugate to $L$. To see this, let $x\in K$ be such that $K=\la K_0x \ra.$
		Note that $K_0=L_0$ is normal in $N_G(Q),$ so we only need to show that $x$ is $N_G(Q)$-conjugate to an element in $L.$
		If $\la x \ra \cap Q\ne 1$, then there exists $g\in N_G(Q)$ such that $\la  x^g \ra \cap Q \leqs C_{Q}(\psi ),$ and hence $x^g\in C_G(\la x^g \ra\cap Q)=Q{:}\la \psi \ra \leqs L.$ Thus, we assume $\la x\ra\cap Q=1$, so $|x|$ is odd. Moreover, by replacing $x$ with a suitable $N_G(Q)$-conjugate, we may assume $x$ is contained in the Hall $2'$-subgroup $D{:}\la \psi \ra$ of $N_G(Q).$ So we may write $x=y\psi^i$ for some element $y\in D$ and integer $i$. Let $s=|\psi^i|$ and $q=q_1^s.$ Then $x^s\in K_0=L_0$ implies that
		$$|x^s|=|yy^{\psi^i}\dots y^{\psi^{(s-1)i}}|=|y^{\frac{q-1}{q_1-1}}|$$ divides $(q^{1/2}+1,q_1-1)=q_1^{1/2}+1,$ and hence
		$$y^{\frac{(q^{1/2}-1)(q^{1/2}+1)}{q_1^{1/2}-1}}=1.$$
		Note that $\frac{q^{1/2}-1}{q_1^{1/2}-1}$ and $q^{1/2}+1$ are coprime, and hence there exist two integers $k_1$ and $k_2$ such that $k_1\frac{q^{1/2}-1}{q_1^{1/2}-1}+k_2(q^{1/2}+1)=1$. Let $z=y^{-k_2(q^{1/2}+1)}.$ Then $(yz)^{q^{1/2}+1}=1,$ and  $$z^\frac{q-1}{q_1-1}=z^{\frac{(q^{1/2}-1)(q^{1/2}+1)}{(q_1^{1/2}-1)(q_1^{1/2}+1)}}=zz^{\psi^i}\dots z^{\psi^{(s-1)i}}=1.$$ Thus, there exists $w\in D$ such that $z=w^{-1}w^{\psi^{-i}}$ due to Hilbert Theorem 90, and so $$x^{w}=(y\psi^i)^{w}=yw^{-1}w^{\psi^{-i}}\psi^i=yz\psi^i\in L.$$ This verifies the claim, so $K$ is a fixer only if $|G:G_0|$ is odd and $K = L$.
		
		Finally, assume $|G:G_0| = |\psi|$ is odd, and we show that $K = L$ is a fixer of $G$. To see this, we may assume that $H_0$ is the subgroup of $\SL_2(q)$ comprising all the invertible $2\times 2$ matrices over $\mathbb{F}_{q^{1/2}}$. Note that $C{:}\la\psi\ra\leqs H$ and $z\in H$, where $z$ is the element defined in \eqref{e:SL_invo}. Now we argue as in the proof of Lemma \ref{l:PSL2_(a)} and show that every $x\in K$ is $G$-conjugate to an element in $H$. If $x\in K$ is of odd order, then $x$ is $K$-conjugate to an element in a Hall $2'$-subgroup $C{:}\la\psi\ra$ of $K$, and so $x^G\cap H\ne \emptyset$. And if $|x|$ is even, then by Lemma \ref{l:PSL2_even_conju}, $x$ is $G$-conjugate to an element in $\la z\ra\times\la\psi\ra\leqs H$. This shows that $K = L$ is a fixer of $G$, which concludes the proof.
	\end{proof}

	\subsection{Case (d)}
	
	Finally, we deal with Proposition \ref{p:PSL2_summarise}(d). Throughout, we write $q_0 = q^{1/r}$ and we recall that $G = G_0.O$ for some $O\leqs\Out(G_0) = \la\delta\ra\times\la\phi\ra$. We may identify $H\cap \PGL_2(q)$ with the image of the subgroup of $\GL_2(q)$ consisting of all the invertible $2\times 2$ matrices over $\mathbb{F}_{q_0}$. That is, $H = C_G(\phi^{f/r})$. Thus, we have $D\cap H\cong C_{q_0-1}$, where $D$ is defined in \eqref{e:D}.
	
	\begin{lem}
		\label{l:PSL2_(d)_q0}
		Assume Proposition \ref{p:PSL2_summarise}(d). Then $|K_0|_{p'}$ divides $(q_0-1)/(2,q-1)$.
	\end{lem}
	
	\begin{proof}
		Note that
		\begin{equation*}
		\left(\frac{q_0+1}{(2,q-1)},\frac{q-1}{(2,q-1)}\right)\text{ divides }\left(\frac{q+1}{(2,q-1)},\frac{q-1}{(2,q-1)}\right) = 1,
		\end{equation*}
		and $K_0$ has an element of order $|K_0|_{p'}$. Now apply Lemma \ref{l:PSL2_ele_cyc} for $H_0\cong\PSL_2(q_0)$.
	\end{proof}

	%
	%
	%
	%
	%
	%
	%
	%
	
	First, we consider the case where $r\ne p$. Define the following group
	\begin{equation}
	\label{e:LI}
	L^{\rm I}:=Q{:}(H\cap (D{:}\la\phi\ra)) \cong (C_p^f{:}C_{\frac{q_0-1}{(2,q-1)}}).O.
	\end{equation}
	Note that $|L^{\rm I}| > |H|$ since $r\geqs 3$.

	\begin{lem}
		\label{l:PSL2_d_I}
		Assume Proposition \ref{p:PSL2_summarise}(d), where $r\ne p$. Then $K$ is a fixer of $G$ if and only if $K$ is $G$-conjugate to a subgroup of $L^{\rm I}$.
	\end{lem}
	
	\begin{proof}
		It suffices to consider the case where $G = \PGammaL_2(q)$.
		We first show that $L^{\rm I}$ is indeed a fixer of $G$. Note that $(H\cap D){:}\la\phi\ra\leqs H$, so by Lemma \ref{l:pre_AGammaL_conju_sub}, if $x\in L^{\rm I}$ and $\la x\ra \cap Q = 1$ then $x^G\cap H\ne\emptyset$. Thus, we may assume that $x\in L^{\rm I}$ and $|\la x\ra\cap Q| = p$. Let $m$ be an integer such that $1\ne x^m\in Q$, and write $y = x^m$. Then by considering a suitable $N_G(Q)$-conjugate of $x$, we may assume $y\in C_G(\phi)$ and so $x\in C_{N_G(Q)}(y) \leqs Q{:}\la\phi\ra$. That is, $\rho(x) = (a,1)\phi^i$ for some integer $i$ and some $a\in\mathbb{F}_q$ (recall that the isomorphism $\rho$ is defined in \eqref{e:pi}). Let $s = |\phi^i|$ and write $q = q_1^s$.
		Note that $r\ne p$, so by Lemma \ref{l:pre_AGammaL_conju_subfield}, there exists $b\in\mathbb{F}_{q_0}$ such that $(a,1)\phi^i$ is conjugate to $(b,1)\phi^i$ by an element in $\Gamma$. This implies that $x$ is conjugate to $\rho^{-1}((b,1)\phi^i)$ by an element in $N_G(Q)$, and so $x$ is $G$-conjugate to an element in $H$ by noting that $\rho^{-1}((b,1)\phi^i)\in H$.
		
		To complete the proof, we show that any fixer $K\leqs N_G(Q)$ is $G$-conjugate to a subgroup of $L^{\rm I}$. Let $K_1:=K\cap \PGL_2(q)$. Then $K = \la K_1,x\ra$ for some $x\in K$. Note that $K_1\leqs Q{:}(H\cap D)\normeq N_G(Q)$. Thus, it suffices to show that $x^g\in L^{\rm I}$ for some $g\in N_G(Q)$, noting that $K_1^g\leqs Q{:}(H\cap D)\leqs L^{\rm I}$. If $|\la x\ra\cap Q| = p$, then for $y\in\la x\ra\cap Q$ of order $p$, $y^g\in C_Q(\phi)$ for some $g\in D$, and hence
		\begin{equation*}
		x^g\in C_{N_G(Q)}(y^g) \leqs Q{:}\la\phi\ra \leqs L^{\rm I}.
		\end{equation*}
		Now assume $|\la x\ra\cap Q| = 1$, so $x$ is $N_G(Q)$-conjugate to an element in $D{:}\la\phi\ra$ by Lemma \ref{l:pre_AGammaL_conju_sub}. Thus, we may assume $\rho(x) = (0,\lambda)\phi^i$ for some integer $i\in\{0,\dots,f-1\}$ and some $\lambda\in\mathbb{F}_q^\times$. If $C:=\la x\ra\cap D = 1$, then $i\ne 0$ and $x^s = 1$, where $s = |\phi^i|$, so
		\begin{equation*}
		(0,1) = \rho(x)^s = ((0,\lambda)\phi^i)^s = (0,\lambda\lambda^{\phi^i}\cdots\lambda^{\phi^{i(s-1)}}) = (0,\mathbf{N}_{\mathbb{F}_q/\mathbb{F}_{q_1}}(\lambda)),
		\end{equation*}
		where $q = q_1^s$ (recall that $\mathbf{N}_{\mathbb{F}_q/\mathbb{F}_{q_1}}(\lambda)$ is the norm of $\lambda$ in the field extension $\mathbb{F}_q/\mathbb{F}_{q_1}$). By Hilbert Theorem 90, there exists $\mu\in\mathbb{F}_q^\times$ such that $\lambda = \mu^{\phi^{-i}}\mu^{-1}$, and thus
		\begin{equation*}
		(\phi^i)^{(0,\mu)} = (0,\mu^{-1})\phi^i(0,\mu)\phi^{-i}\phi^i = (0,\mu^{\phi^{-i}}\mu^{-1})\phi^i = (0,\lambda)\phi^i = \rho(x).
		\end{equation*}
		This implies that $x$ is $D$-conjugate to $\phi^i\in L^{\rm I}$.
		
		Finally, assume $x\in D{:}\la\phi\ra$ and $|C|\ne 1$. Since $K$ is a fixer of $G$ and $x\in K$, there exists $g\in G$ such that $x^g\in H$, so $x^g\in N_G(C^g)$. Note that any two subgroups of $H\cap \PGL_2(q)$ isomorphic to $C_m$ for $m > 2$ are $H$-conjugate. In fact, if $C\cong C_2$, then it lies in $H_0$ if and only if $C^g\leqs H_0$, so they are also $H$-conjugate. This implies that $C = C^{gh_1}$ for some $h_1\in H$, which yields $x^{gh_1}\in N_H(C)$ and $gh_1\in N_G(C)$. Observe that $N_G(C) = D{:}\la z,\phi\ra$ is a maximal subgroup of $G$ of type $\GL_1(q)\wr S_2$, where
		\begin{equation*}
		z =
		\begin{pmatrix}
		0&1\\
		-1&0
		\end{pmatrix}
		Z(\GL_2(q)),
		\end{equation*}
		and hence $N_G(C) = DN_H(C)$ since $z,\phi\in H$. This implies that there exists $h_2\in H$ such that $gh_1h_2\in D\leqs N_G(Q)$ and $x^{gh_1h_2}\in N_G(Q)\cap H \leqs L^{\rm I}$. This completes the proof.
	\end{proof}
	
	Next, we turn to the case where $r = p$ and $(p,|K:K_0|) = 1$. Here we define
	\begin{equation}
	\label{e:LII}
	L^{{\rm II}}:=Q{:}(H\cap (D{:}\la\phi\ra_{p'})) \cong (C_p^f{:}C_{\frac{q_0-1}{(2,q-1)}}).O_{p'}.
	\end{equation}

	\begin{rem}
		\label{r:LII_order}
		We remark that if $r = p$ then $|L^{\rm II}| > |H|$ for all $q > 61$. To see this, first note that
		\begin{equation*}
		\frac{|H|}{|L^{\rm II}|} = \frac{q_0(q_0+1)}{q}\cdot |O_p|\leqs \frac{q_0(q_0+1)}{q}\cdot f\leqs \frac{q_0(q_0+1)}{q}\cdot \log_2q.
		\end{equation*}
		Thus, if $p \geqs 5$ then we have
		\begin{equation*}
		\frac{|H|}{|L^{\rm II}|} \leqs \frac{2\log_2q}{q^{3/5}},
		\end{equation*}
		which is less than $1$ for all $q > 61$. Now assume $p = 3$, so $q_0 = q^{1/3}$ is a $3$-power and
		\begin{equation*}
		\frac{|H|}{|L^{\rm II}|} \leqs \frac{3(q_0+1)}{q_0^2}\cdot \log_3 q_0.
		\end{equation*}
		The latter term is less than $1$ if $q_0\geqs 27$, and one can also check that $|L^{\rm II}| > |H|$ if $q_0 = 9$. Note that if $q_0 = 3$ then $q = 27$ and $|L^{\rm II}| < |H|$.
	\end{rem}

	\begin{lem}
		\label{l:PSL2_d_II}
		Assume Proposition \ref{p:PSL2_summarise}(d), where $r = p$ and $(p,|K:K_0|) = 1$. Then $K$ is a fixer of $G$ if and only if $K$ is $G$-conjugate to a subgroup of $L^{\rm II}$.
	\end{lem}
	
	\begin{proof}
		Once again, we only need to consider the case where $G = \PGammaL_2(q)$, and we argue similarly as in the proof of Lemma \ref{l:PSL2_d_I}. First, we show that $L^{\rm II}$ is a fixer of $G$. Note that if $x\in L^{\rm II}$ and $\la x\ra\cap Q = 1$ then $|x|$ is coprime to $p$ and hence $x$ is $N_G(Q)$-conjugate to an element in a Hall $p'$-subgroup $H\cap (D{:}\la\phi\ra_{p'})\leqs H$ of $L^{\rm{II}}$, so $x^G\cap H\ne \emptyset$. Now assume $x\in L^{\rm II}$ and $|\la x\ra\cap Q| = p$, so $y:=x^m\in Q$ for some integer $m < |x|$. It follows that there exists $z\in N_G(Q)$ such that $y^z\in C_Q(\phi)$ and $x^z\in C_{N_G(Q)}(y^z)\leqs Q{:}\la\phi\ra$. Thus, $\rho(x^z) = (a,1)\phi^i$ for some integer $i$ and some element $a\in\mathbb{F}_q$. Note that $x\in L^{\rm II}$ implies that $r\nmid |\phi^i|$, so we can apply Lemma \ref{l:pre_AGammaL_conju_subfield} once again, which shows that $x^z$ is $N_G(Q)$-conjugate to $\rho^{-1}((b,1)\phi^i)$ for some $b\in\mathbb{F}_{q_0}$, noting that $\rho^{-1}((b,1)\phi^i)\in H$.
		
		Now let $K\leqs N_G(Q)$ be a fixer of $G$ such that $(p,|K:K_0|) = 1$, and we will show that $K$ is $G$-conjugate to a subgroup of $L^{\rm II}$. Let
		\begin{equation*}
		R:=N_{\PGL_2(q)}(Q).\la\phi\ra_{p'}= N_G(Q)\cap (\PGL_2(q).\la\phi\ra_{p'}).
		\end{equation*}
		Then $K\leqs R$ and $K = \la K_1,x\ra$ for some $x\in K$, where $K_1 := K\cap \PGL_2(q)$, noting that $K_1\leqs Q{:}(H\cap D)\normeq N_G(Q)$. Since $K_1^g\leqs Q{:}(H\cap D)\leqs L^{\rm II}$ for any $g\in N_G(Q)$, it suffices to show that $x$ is $N_G(Q)$-conjugate to an element of $L^{\rm II}$. If $|\la x\ra\cap Q| = p$ then for any $y\in\la x\ra\cap Q$ of order $p$ we have $y^g\in C_Q(\phi)$ for some $g\in D$, which yields
		\begin{equation*}
		x^g\in C_R(y^g)\leqs Q{:}\la\phi\ra_{p'}\leqs L^{\rm II}.
		\end{equation*}
		This allows us to assume $|\la x\ra\cap Q| = 1$, so $|x|$ is coprime to $p$, and hence $x$ is $N_G(Q)$-conjugate to an element in a Hall-$p'$ subgroup $D{:}\la\phi\ra_{p'}$ of $N_G(Q)$. One can use the same arguments as in the proof of Lemma \ref{l:PSL2_d_I} to show that $x$ is $D$-conjugate to an element in $R\cap H\leqs L^{\rm II}$, so we omit the details.
	\end{proof}
	
	\begin{rem}
		\label{r:LI_LII}
		We remark that not every subgroup of $G$ of the form $(C_p^f{:}C_{\frac{q_0-1}{(2,q-1)}}).O$ is isomorphic to $L^{\rm{I}}$. That is, there might exist a subgroup $K$ of $G$ such that $K_0 = L_0^{\rm I}$ and $K/K_0 = L^{\rm I}/L^{\rm I}_0$, while $K\not\cong L^{\rm I}$. Similarly, not every subgroup of $G$ of the form $(C_p^f{:}C_{\frac{q_0-1}{(2,q-1)}}).O_{p'}$ is isomorphic to $L^{\rm{II}}$. To see this, assume $G = \PGammaL_2(q)$, $p$ is odd and $r$ is a prime divisor of $p-1$. Then $|\phi| = r$, $H\cong \PGL_2(p)\times C_r$ and we have $L^{\rm I} = L^{\rm II} = Q{:}((D\cap H)\times \la\phi\ra)$. Let $x$ be a generator of $D$, and define $K = \la Q{:}(D\cap H),x\phi\ra$. Then $K$ is isomorphic to a group of the form $(C_p^f{:}C_{\frac{q_0-1}{(2,q-1)}}).O$. Note that
		\begin{equation*}
		(x\phi)^r = xx^{\phi^{-1}}x^{\phi^{-2}}\cdots x^{\phi^{-(r-1)}} = x^{(p^r-1)/(p-1)},
		\end{equation*}
		and so $|x\phi| = r(p-1)$. However, $L^{\rm I} = L^{\rm II}$ does not contain an element of order $r(p-1)$, whence $L^{\rm I}\not\cong K$.
	\end{rem}
	
	Finally, let us consider the case where $r = p$ divides $|K:K_0|$. Define
	\begin{equation*}
	M:=\left\{xZ(\GL_2(q)):x =
	\begin{pmatrix}
	1&a\\
	&1
	\end{pmatrix}
	\mbox{ for some }a\in\mathbb{F}_q\mbox{ with }\Tr_{\mathbb{F}_q/\mathbb{F}_{q_0}}(a) = 0
	\right\}\leqs Q,
	\end{equation*}
	noting that $M\cong C_p^{f-f/p}$. Let
	\begin{equation}
	\label{e:LIII}
	L^{\rm III} := M{:}(H\cap (D{:}\la\phi\ra))\cong C_p^{f-f/p}{:}C_{\frac{q_0-1}{(2,q-1)}}.O.
	\end{equation}
	
	\begin{rem}
		\label{r:LIII_order}
		Note that
		\begin{equation*}
		\frac{|H|}{|L^{\rm III}|} = \frac{q_0^2(q_0+1)}{q}
		\end{equation*}
		if $r = p$. Thus, it is easy to see that $|L^{\rm III}| > |H|$ if and only if $p > 3$.
	\end{rem}
	
	\begin{lem}
		\label{l:PSL2_d_III}
		Assume Proposition \ref{p:PSL2_summarise}(d), where $r = p$ divides $|K:K_0|$. Then $K$ is a fixer of $G$ if and only if $K$ is $G$-conjugate to a subgroup of $L^{\rm III}$.
	\end{lem}
	
	\begin{proof}
		We first show that $L^{\rm III}$ is a fixer of $G$, and again, we may assume $G = \PGammaL_2(q)$. Let $x\in L^{\rm III}$. If $x\in L^{\rm II}$, then by arguing as in the proof of Lemma \ref{l:PSL2_d_II}, we see that $x$ is $N_G(Q)$-conjugate to an element  in $H$. Thus, we may assume $x\notin L^{\rm II}$. Suppose $\la x\ra_p\leqs M$. Note that $\la x\ra_{p'}$ lies in a Hall $p'$-subgroup of $C_{L^{\rm III}}(\la x\ra_p)$, which is contained in a Hall $p'$-subgroup of $L^{\rm II}$. Thus, $\la x\ra = \la x\ra_p\la x\ra_{p'}\leqs L^{\rm II}$, a contradiction.
		
		Hence, it suffices to consider the case where $\la x\ra_p$ is not a subgroup of $M$.
		Since $M{:}\la \phi\ra_p$ is a Sylow $p$-subgroup of $L^{\rm III}$, we may also assume $\la x\ra_p\leqs M{:}\la \phi\ra_p$. Let $y$ be a generator of $\la x\ra_p$. Then $\rho(y) = (a,1)\phi^{i}$ for some $a\in\mathbb{F}_q$ with $\Tr_{\mathbb{F}_q/\mathbb{F}_{q_0}}(a) = 0$, and we see that $p\mid |\phi^i|$ since $y\in M{:}\la \phi\ra_p$ and $y\notin M$. Let $s = |\phi^i|$ and $q_1^s = q$. Then $\Tr_{\mathbb{F}_q/\mathbb{F}_{q_1}}(a) = 0$ by Lemma \ref{l:pre_AGammaL_trace}(i).
		Now Lemma \ref{l:pre_AGammaL_conju}(i) implies that $\rho(y)$ is $\mathbb{F}_q^+$-conjugate to $(0,1)\phi^i$, and thus $y$ is $Q$-conjugate to $\phi^i$. It follows that $x$ is $Q$-conjugate to an element in $C_G(\phi^i)\leqs C_G(\phi^{f/p}) = H$. We conclude that $L^{\rm III}$ is a fixer of $G$.
		
		To complete the proof, let $K\leqs N_G(Q)$ be a fixer of $G$, and we need to show that $K$ is $G$-conjugate to a subgroup of $L^{\rm III}$. Let $x\in K$ be such that $K=\la K_1, x \ra,$ where $K_1=K\cap \PGL_2(q)$.
		By our assumption, we have $p$ divides $|G_0x|$, which divides $|x|$. Let $y$ be a generator of the Sylow $p$-subgroup $\la x\ra_p$ of $\la x\ra$. Suppose $\la x\ra\cap Q\ne 1$. Then there exists $g_1\in D$ such that $\la x^{g_1}\ra\cap Q\leqs C_{N_G(Q)}(\phi)$, and thus
		\begin{equation*}
		y^{g_1}\in C_{N_G(Q)}(\la x^{g_1}\ra\cap Q)\leqs Q{:}\la\phi\ra.
		\end{equation*}
		In particular, we have $y^{g_1}\leqs Q{:}\la\phi\ra_p$ since $|y|$ is a $p$-power. Note that $y\in K$ and $K$ is a fixer of $G$, so there exists $g_2\in G$ such that $y^{g_1g_2}\in (Q\cap H){:}\la\phi\ra_p$. It follows that $\rho(y^{g_1g_2}) = (a,1)\phi^i$ for some $a\in\mathbb{F}_{q_0}$ with $p\mid |\phi^i|=:s$. Write $q = q_1^s$, so we have $\mathbb{F}_{q_1}\subseteq \mathbb{F}_{q_0}$ and $\Tr_{\mathbb{F}_q/\mathbb{F}_{q_1}}(a) = 0$ by Lemma \ref{l:pre_AGammaL_trace}(i). Now Lemma \ref{l:pre_AGammaL_conju}(i) implies that $\rho(y^{g_1g_2})$ is $\mathbb{F}_q^+$-conjugate to $(0,1)\phi^i$, and so $y^{g_1g_2}$ is $Q$-conjugate to $\phi^i$. Therefore, $y^{g_1g_2g_3} = \phi^i$ for some $g_3\in Q$, which yields
		\begin{equation*}
		1\ne (\la x\ra \cap Q)^{g_1g_2g_3} = (\la y\ra \cap Q)^{g_1g_2g_3} \leqs (\la y\ra\cap G_0)^{g_1g_2g_3} = \la y^{g_1g_2g_3}\ra\cap G_0 = \la\phi^i\ra\cap G_0 = 1,
		\end{equation*}
		a contradiction.
		
		This implies that $\la x\ra\cap Q = 1$, and so there exists $h_1\in N_G(Q)$ such that $y^{h_1}\in\la\phi\ra_p$ by Lemma \ref{l:pre_AGammaL_conju_sub}, noting that $\la\phi\ra_p$ is a Sylow $p$-subgroup of $H\cap (D{:}\la\phi\ra)$. Let $y^{h_1} = \phi^i$ and let $s = |\phi^i|$. Then $p\mid s$ and $\mathbb{F}_{q_1}\subseteq\mathbb{F}_{q_0}$, where $q_1^s = q$.
		It is easy to see that
		\begin{equation*}
		\rho(x^{h_1})\in \rho(C_{N_G(Q)}(y^{h_1})) = \rho(C_{N_G(Q)}(\phi^i)) = (\mathbb{F}_{q_1}^+{:}\mathbb{F}_{q_1}^\times){:}\la \phi\ra.
		\end{equation*}
		Moreover, if $z$ is a generator of $\la x\ra_{p'}$, then $\rho(z^{h_1h_2})\in \mathbb{F}_{q_1}^\times{:}\la\phi\ra_{p'}$ for some $h_2\in C_{N_G(Q)}(y^{h_1})$, and so
		\begin{equation*}
		\la x^{h_1h_2}\ra = \la y^{h_1h_2}\ra\la z^{h_1h_2}\ra\leqs (H\cap D){:}\la\phi\ra.
		\end{equation*}
		Note that $\phi^{f/p}\in\la y^{h_1h_2}\ra\leqs K^{h_1h_2}$, and for any $w\in K^{h_1h_2}\cap Q$, we have $w\phi^{f/p}\in K^{h_1h_2}$ is $G$-conjugate to an element in
		\begin{equation*}
		H\cap (\PGL_2(q).\la\phi^{f/p}\ra)\cong \PGL_2(q_0)\times \la \phi^{f/p}\ra.
		\end{equation*}
		This implies that $|w\phi^{f/p}| = p$. Now if $w\notin M$, then $\rho(w\phi^{f/p}) = (a,1)\phi^{f/p}$ for some $a\in\mathbb{F}_q$ with $\Tr_{\mathbb{F}_q/\mathbb{F}_{q_0}}(a)\ne 0$, so
		\begin{equation*}
		|w\phi^{f/p}| = |(a,1)\phi^{f/p}| \ne |(0,1)\phi^{f/p}| = p
		\end{equation*}
		by Lemma \ref{l:pre_AGammaL_conju}(iii), which gives a contradiction. It follows that $K^{h_1h_2}\cap Q\leqs M$.
		
		Therefore, conjugating by an element in $G$ if necessary, we may assume $x\in (H\cap D){:}\la \phi\ra$ and $K\cap Q\leqs M$. Once again, we write $K = \la K_1,x\ra$, where $K_1 = K\cap \PGL_2(q)$. Noting that $K_1 = (K\cap M){:}C^g$ for some $g\in Q$ and $C\leqs (D\cap H)$. Then we have
		\begin{equation*}
		C^{g^{x^{-1}}} = C^{xgx^{-1}} = C^{x^{-1}xgx^{-1}} = C^{gx^{-1}} = C^{gh} = C^{hg}
		\end{equation*}
		for some $h\in K\cap M$. This is because $x\in N_G(K_1)$ implies that $C^{gx^{-1}}\leqs K_1$, and every subgroup of $K_1$ isomorphic to $C$ is $(K\cap M)$-conjugate to $C^g$ as $C^g$ is a Hall $p'$-subgroup of $K_1$. Thus, $g^{x^{-1}}(hg)^{-1}\in C_Q(C) = 1$, which yields
		\begin{equation*}
		x^{g^{-1}} = gxg^{-1}x^{-1}x = g(g^{-1})^{x^{-1}}x = h^{-1}x\in (M{:}C){:}\la\phi\ra.
		\end{equation*}
		Finally, we observe that
		\begin{equation*}
		K_1^{g^{-1}} = (K\cap M)^{g^{-1}}{:}C^{gg^{-1}} = (K\cap M){:}C\leqs (M{:}C){:}\la \phi\ra.
		\end{equation*}
		We conclude the proof by noting that $K^{g^{-1}} = \la K_1^{g^{-1}},x^{g^{-1}}\ra\leqs (M{:}C){:}\la\phi\ra\leqs L^{\rm III}$.
	\end{proof}

	\begin{lem}
		\label{l:PSL2_d_III_rem}
		Suppose $q = q_0^p$ and let $K\leqs G$ be such that $K_0\cong C_p^{f-f/p}{:}C_{(q_0-1)/2}$ and $p$ divides $|K:K_0|$, then $K$ is $G$-conjugate to a subgroup of $L^{\rm III}$.
	\end{lem}
	
	\begin{proof}
		First note that $K_0$ has a unique normal Sylow $p$-subgroup. Thus, by considering a suitable $G$-conjugate, we may assume $K\leqs N_G(Q)$. Let $x\in K$ be such that $|x| = p^m$ for some $m$ and $|K_0x| = p$. Then $x\in G_0.\la\phi^{f/p}\ra$. It follows that $x$ is $G$-conjugate to an element in $Q{:}\la\phi^{f/p}\ra_p$ since it is a Sylow $p$-subgroup of $G_0.\la\phi^{f/p}\ra$. This allows us to further assume $x = g\phi^{f/p}\in K$ for some $g\in Q$.
		
		Next, we show that the unique Sylow $p$-subgroup $K\cap Q$ of $K_0$ is no other than $M$. Let $C$ be a Hall $p'$-subgroup of $K_0$. Then $\rho(C) = \la (a,\lambda)\ra$ for some $a\in\mathbb{F}_q^+$ and $\lambda\in\mathbb{F}_{q_0}^\times$. Moreover, observe that $\la\lambda\ra = \mathbb{F}_{q_0}^\square$ is the group of all non-zero square elements in $\mathbb{F}_{q_0}$ under multiplication. Note that $(b,1)^{(a,\lambda)} = (\lambda b,1)$ for any $b\in\mathbb{F}_q$, whence the set
		\begin{equation*}
		S = \{c\in\mathbb{F}_q : (c,1)\in\rho(K\cap Q)\}
		\end{equation*}
		is closed under the multiplication of $\mathbb{F}_{q_0}^\square$ as $\rho(K\cap Q)\normeq \rho(K)$. In fact, one can show that $S$ is a $(p-1)$-dimensional $\mathbb{F}_{q_0}$-subspace of $\mathbb{F}_q$. Moreover, note that $\rho(x) = (b,1)\phi^{f/p}$ for some $b\in\mathbb{F}_q^+$, and we have
		\begin{equation*}
		(c,1)^{\rho(x)} = (c,1)^{(b,1)\phi^{f/p}} = (c^{\phi^{f/p}},1).
		\end{equation*}
		Thus, $S$ is also invariant under $\phi^{f/p}$. Since $\phi^{f/p}$ fixes a unique $1$-dimensional $\mathbb{F}_{q_0}$-subspace of $\mathbb{F}_q$, it is a regular unipotent element on $\mathbb{F}_q$ as an $\mathbb{F}_{q_0}$-linear map. It follows that $\phi^{f/p}$ fixes a unique $(p-1)$-dimensional $\mathbb{F}_{q_0}$-subspace of $\mathbb{F}_q$. Now observe that
		\begin{equation*}
		\{c\in\mathbb{F}_q : \Tr_{\mathbb{F}_q/\mathbb{F}_{q_0}}(c) = 0\}
		\end{equation*}
		is a $(p-1)$-dimensional $\mathbb{F}_{q_0}$-subspace fixed by $\phi^{f/p}$. Therefore, we have
		\begin{equation*}
		S = \{c\in\mathbb{F}_q : \Tr_{\mathbb{F}_q/\mathbb{F}_{q_0}}(c) = 0\}
		\end{equation*}
		and we conclude that $K\cap Q = M$.
		
		Hence, we have $K_0 = M{:}C$ with $|C| = (q_0-1)/2 = |D_0\cap H|$. Then there exists $g\in Q$ such that $C^g = D_0\cap H$, so $K_0^g = M^g{:}C^g = M{:}(D_0\cap H)=:E$, whence $K^g\leqs N_G(K_0^g) = N_G(E)$. It suffices to show that $N_G(E)\leqs L^{\rm III}$, and we claim that $N_G(E) = L^{\rm III}$. To see this, first note that $Q\leqs N_{G_0}(M)$, and the order of a Hall $p'$-subgroup of $N_{G_0}(M)$ divides
		\begin{equation*}
		|D_0\cap H| = (p^{f-f/p}-1,(p^f-1)/2)
		\end{equation*}
		by Lemma \ref{l:PSL2_sub}. It follows that $N_{G_0}(M) = Q{:}(D_0\cap H)$. Observe that both $Q{:}(D_0\cap H)$ and $E$ are Frobenius groups. Thus, we have $N_{Q}(E) = M$, and hence $E$ is self-normalising in $G_0$. This implies that
		\begin{equation*}
		|N_G(E)|\leqs |N_{G_0}(E)|\cdot |G:G_0| = |E|\cdot |G:G_0|\leqs L^{\rm III}.
		\end{equation*}
		On the other hand, it is evident that $L^{\rm III}\leqs N_G(E)$, which yields $L^{\rm III} = N_G(E)$ as claimed. This shows that $K^g\leqs L^{\rm III}$ as explained above.
	\end{proof}

	Now we are in the position to prove Theorem \ref{thm:PSL2}.
	
	\begin{proof}[Proof of Theorem \ref{thm:PSL2}.]
		We refer the reader to Proposition \ref{p:PSL2_61} for the case where $q\leqs 61$, so we may assume $q > 61$. Assume $H$ is of type $2^{1+2}_-.\mathrm{O}_2^-(2)$ or $A_5$. Then Lemma \ref{l:PSL2_C6_S} gives a classification of large fixers. If $H$ is of type $P_1$ or $\GL_1(q^2)$, then $G$ has no large fixer by Lemmas \ref{l:PSL2_para} and \ref{l:PSL2_C3}.
		
		To complete the proof, assume $H$ is of type $\GL_1(q)\wr S_2$ or $\GL_2(q_0)$ with $q = q_0^r$. Then Proposition \ref{p:PSL2_summarise} shows that only four cases arise, and cases (a), (b) and (c) are handled in Lemmas \ref{l:PSL2_(a)}, \ref{l:PSL2_(b)} and \ref{l:PSL2_(c)}, respectively.
		
		Finally, assume Proposition \ref{p:PSL2_summarise}(d), so $H$ is of type $\GL_2(q^{1/r})$ with $r$ odd, and $K$ is a subgroup of a maximal parabolic subgroup of $G$. The case where $r\ne p$ is treated in Lemma \ref{l:PSL2_d_I}. If $r = p$ and $(p,|K:K_0|) = 1$ then $K$ is a fixer if and only if $K$ is $G$-conjugate to a subgroup of $L^{\rm II}$ defined as in \eqref{e:LII} (see Lemma \ref{l:PSL2_d_II}), and as remarked in Remark \ref{r:LII_order}, we see that $|L^{\rm II}| > |H|$. Now if $r = p$ divides $|K:K_0|$, then by Lemma \ref{l:PSL2_d_III}, $K$ is a fixer of $G$ if and only if $K$ is $G$-conjugate to a subgroup of $L^{\rm III}$ defined as in \eqref{e:LIII}. However, if $p = 3$ then $|H| > |L^{\rm III}|$ (see Remark \ref{r:LIII_order}), so this case is excluded in Table \ref{tab:PSL2}. The associated condition recorded in Table \ref{tab:PSL2} is sufficient in view of Lemma \ref{l:PSL2_d_III_rem}.
	\end{proof}
	
	As an immediate application, we establish Corollaries \ref{cor:PSL2_weak} and \ref{cor:rho_0_PSL2}.
	
	\begin{proof}[Proof of Corollary \ref{cor:PSL2_weak}]
		If $H$ is not of type $\GL_2(q^{1/r})$ then the statement is deduced straightforwardly by inspecting Table \ref{tab:PSL2}, so we may assume $H$ is of type $\GL_2(q^{1/r})$.
		
		First assume $r = p = 2$. By Lemma \ref{l:PSL2_(c)}, $G$ has a large fixer if and only if $|G:G_0|$ is odd. Thus, $G$ has the weak-EKR property if and only if $|G:G_0|$ is even. One can obtain the same argument when considering the strict-weak-EKR property.
		
		Now assume $r$ is odd. Here if $r\ne p$ then $L^{\rm I}$ is a large fixer, while $L^{\rm II}$ is a large fixer otherwise (see Lemmas \ref{l:PSL2_d_I} and \ref{l:PSL2_d_II}, respectively). Thus, in either case, $G$ does not have the weak-EKR property.
	\end{proof}
	
	\begin{proof}[Proof of Corollary \ref{cor:rho_0_PSL2}]
		First note that $|H|\sqrt{|\Omega|} = \sqrt{|G|\cdot |H|}$. In view of Proposition \ref{p:PSL2_61}, we may assume $q > 61$, and so one of the cases described in Proposition \ref{p:PSL2_summarise} arises.
		
		Assume Proposition \ref{p:PSL2_summarise}(a). Then
		\begin{equation*}
		\frac{|K|}{\sqrt{|G|\cdot |H|}}\leqs \frac{q(q-1)}{\sqrt{q(q^2-1)\cdot 2(q-1)}} = \frac{1}{\sqrt{2}}\cdot \sqrt{\frac{q}{q+1}} < \frac{1}{\sqrt{2}}.
		\end{equation*}
		Here we note that the bound is sharp. Indeed, if $G = G_0$ and $K$ is a maximal subgroup of $G$ of type $P_1$, then the ratio in the above inequality tends to $1/\sqrt{2}$ as $q\to \infty$. One can check the desired bound for other cases in Proposition \ref{p:PSL2_summarise} easily.
	\end{proof}

	\section{Proof of Theorem \ref{thm:Spiga}}
	
	\label{s:Spiga}
	
	Next, we turn to the proof of Theorem \ref{thm:Spiga}, and we may assume $q > 61$ since the conclusion to Theorem \ref{thm:Spiga} with $q\leqs 61$ can be verified with the aid of {\sc Magma}. Recall that $D(G,H)$ is the set of derangements of $G$ in its coset action on $[G:H]$. The following result follows from Theorem \ref{thm:PSL2}.
	
	\begin{lem}
		\label{l:psl_even_HK}
		Suppose $\soc(G) = \PSL_2(q)$, $q > 61$ and $H,K$ are maximal subgroups of $G$ with $|K|\geqs |H|$. Then $D(G,H) = D(G,K)$ if and only if one of the following holds:
		\begin{itemize}\addtolength{\itemsep}{0.2\baselineskip}
			\item[{\rm (a)}] $q$ is even, $H$ is of type $\GL_1(q)\wr S_2$, $K$ is of type $P_1$ and $[G:G_0]$ is odd;
			\item[{\rm (b)}] $G = G_0$, $q=p\equiv \pm1\pmod 8$, $H$ is of type $2_-^{1+2}.\mathrm{O}_2^-(2)$ and $K = H^\delta$;
			\item[{\rm (c)}] $G = G_0$, $q=p\equiv \pm1\pmod {10}$, $H$ is of type $A_5$ and $K = H^\delta$;
			\item[{\rm (d)}] $G = G_0$, $q = p^2$, $3\ne p\equiv\pm 3\pmod {10}$, $H$ is of type $A_5$ and $K = H^\delta$.
		\end{itemize}
	\end{lem}
	
	\begin{proof}
		In view of Theorem \ref{thm:PSL2}, it suffices to show that $D(G,K) \subseteq  D(G,H)$ in case (a). Let $x\in H$. If $|x|$ is odd, then $x$ is in a Hall $2'$-subgroup of $H$, which is $G$-conjugate to a Hall $2'$-subgroup of $K$. Now assume $|x|$ is even. Note that $|xG_0|$ is odd, which allows us to apply Lemma \ref{l:PSL2_even_conju}, yielding that $x$ is $G$-conjugate to an element of $C_G(z)$. It follows that $x$ is $G$-conjugate to an element of $C_G(u)\leqs K$, where $u\in K$ is a unipotent element. Therefore, $H\subseteq \bigcup_{g\in G}K^g$ and so $D(G,K)\subseteq D(G,H)$.
	\end{proof}

Now we establish Theorem \ref{thm:Spiga}. We thank an anonymous referee for providing the much simpler proof than the previous version.

\begin{proof}[Proof of Theorem \ref{thm:Spiga}]
	It suffices to consider the four cases recorded in Lemma \ref{l:psl_even_HK}.
	Note that in cases (b), (c) and (d), the actions of $G$ on $[G:H]$ and $[G:K]$ are permutation equivalent, so their permutation characters coincide.
	
	With this in mind, we assume case (a) of Lemma \ref{l:psl_even_HK} for the remainder of the proof. We may identify $\Delta:=[G:K]$ with the set of $1$-subspaces of $\mathbb{F}_q^2$, and $\Omega:=[G:H]$ with the set of $2$-subsets of $\Delta$. In addition, the permutation characters of $G$ on $\Delta$ and $\Omega$ are denoted $\pi_\Delta$ and $\pi_\Omega$, respectively (that is, $\pi_\Delta(x) = |\fix_\Delta(x)|$ and $\pi_\Omega(x) = |\fix_\Omega(x)|$). Recall that our aim is to show that $\pi_\Omega-\pi_\Delta$ is a character (a linear combination of complex irreducible characters with positive integer coefficients) of $G$.
	
	Since $G$ is $2$-transitive on $\Delta,$ it follows that $\pi_{\Delta}$ is a sum of the trivial character of $G$ and another irreducible character. Note that $\pi_{\Omega}=\mathbf{1}_H \uparrow G$ is the induced character of the trivial character of $H$, and thus
	\begin{equation*}
	\la \pi_{\Omega}, \pi_{\Delta} \ra_G=\la \mathbf{1}_H \uparrow G, \pi_{\Delta} \ra_G= \la \mathbf{1}_H , \pi_{\Delta} \downarrow H \ra_H =\frac{1}{|H|}\sum_{h\in H}|\fix_\Delta(h)|,
	\end{equation*}
	which coincides with the number of orbits of $H$ acting on $\Delta$ by Burnside's lemma. On the other hand, it is not difficult to see that $H$ has precisely two orbits on $\Delta$.
	Therefore, $\la \pi_{\Omega}, \pi_{\Delta} \ra_G=2,$ and hence both irreducible constituents of $\pi_{\Delta}$ are constituents of  $\pi_{\Omega},$ yielding $\pi_{\Omega}-\pi_{\Delta}$ is a character of $G$.
\end{proof}

	\section{Sporadic groups}
	
	\label{s:spo}
	
	Now we turn to almost simple sporadic groups and we will establish Theorems \ref{thm:spo_rho_0} and \ref{thm:spo} in this section. Recall that $D(G,H)\subseteq D(G,K)$ if and only if $K$ is a fixer of $G$ in its action on $[G:H]$.
	
	\begin{prop}
		\label{p:spo}
		The conclusion to Theorem \ref{thm:spo} holds, and Corollary \ref{cor:rho_1_An} holds for sporadic groups.
	\end{prop}
	
	\begin{proof}
		First assume $G$ is not the Monster group $\mathbb{M}$. The conjugacy classes of $G$, $H$ and $K$ can be read off from the character tables of $G$ and $H$, which can be accessed computationally via the \textsf{GAP} Character Table Library \cite{B_GAPCTL}. In addition, apart from the special case where $G = \mathbb{B}$ and $(2^2\times F_4(2)){:}2\in\{H,K\}$, the fusion maps of $H$-conjugacy classes and $K$-conjugacy classes in $G$ can be obtained using the \textsf{GAP} function \texttt{FusionConjugacyClasses}, so it is easy to check whether $x^G\cap H = \emptyset$ for each $x\in K$. If $G = \mathbb{B}$ and $(2^2\times F_4(2)){:}2\in\{H,K\}$ then $H\not\cong K$ and so $|H|\ne |K|$ since there is a unique $G$-class of such subgroups, and one can check that $\Spec(K)\not\subseteq\Spec(H)$ whenever $|K| > |H|$, which is incompatible with Lemma \ref{l:pre_fix}(i).
		
		To complete the proof, we assume $G = \mathbb{M}$, and we will show that no case arises in this setting. The list $\mathcal{A}$ of maximal subgroups of $G$ is presented in \cite[Table 1]{DLP_M}, and \cite[Theorem 1]{DLP_M} asserts that any two isomorphic maximal subgroups of $G$ are $G$-conjugate, which implies that $|H|\ne |K|$. Let $\mathcal{A}_1\subseteq \mathcal{A}$ be the set of maximal subgroups of $G$ whose character table can be accessed via the \textsf{GAP} function \texttt{NamesOfFusionSources}. Note that for any subgroup $L$ in $\mathcal{A}_1$, the fusion map of $L$-classes can be obtained via \texttt{FusionConjugacyClasses} as above. Thus, if $H,K\in\mathcal{A}_1$ then we can adopt the same method as above.
		
		Now we define another set $\mathcal{A}_2\subseteq\mathcal{A}\setminus\mathcal{A}_1$ of maximal subgroups of $G$, comprising the remaining almost simple subgroups, as well as the subgroups with a permutation representation accessible in \cite{W_WebAt}. It is easy to obtain the spectra of subgroups contained in $\mathcal{A}_1\cup \mathcal{A}_2$, and we can use Lemma \ref{l:pre_fix}(i) to eliminate all cases where $H,K\in\mathcal{A}_1\cup\mathcal{A}_2$ and at least one of these subgroups is in $\mathcal{A}_2$.
		
		Now let us turn to the maximal subgroups lying in $\mathcal{A}_3:=\mathcal{A}\setminus(\mathcal{A}_1\cup\mathcal{A}_2)$, noting that
		\begin{equation*}
		\mathcal{A}_3 = \{2^{10+16}.\POmega_{10}^+(2),2^{5+10+20}.(S_3\times \PSL_5(2)),2^{3+6+12+18}.(\PSL_3(2)\times 3S_6)\}.
		\end{equation*}
		If $H\in\mathcal{A}_3$ then one can see that $\pi(K)\not\subseteq\pi(H)$ for any $K$ with $|K| > |H|$, so Lemma \ref{l:pre_fix}(ii) implies that no case arises. It suffices to consider the cases where $K\in\mathcal{A}_3$.
		
		If $K = 2^{10+16}.\POmega_{10}^+(2)$, so $\pi(K) = \{2,3,5,7,17,31\}$ then it is routine to check that $\pi(K)\not\subseteq\pi(H)$ for any $H$ with $|H| < |K|$, which is incompatible with Lemma \ref{l:pre_fix}(ii).
		
		Next, assume $K = 2^{5+10+20}.(S_3\times \PSL_5(2))$. Then $\pi(K)\subseteq\pi(H)$ and $|K| > |H|$ implies that $H = S_3\times\mathrm{Th}$. Note that $K$ contains a Sylow $2$-subgroup of $G$, and so $32\in\Spec(K)$. However, since $\mathrm{Th}$ has no element of $32$, we have $32\notin\Spec(H)$, which is incompatible with Lemma \ref{l:pre_fix}(i).
		
		Finally, assume $K = 2^{3+6+12+18}.(\PSL_3(2)\times 3S_6)$. Again, $K$ contains a Sylow $2$-subgroup of $G$, so $K$ has an element of order $32$. Here if $H\in\mathcal{A}_1$ then one can check that either $|H| > |K|$ or $32\notin\Spec(H)$ with the aid of \textsf{GAP}. Note that $\pi(K) = \{2,3,5,7\}$, so if $H\notin\mathcal{A}_1$ and $|K| > |H|$ then Lemma \ref{l:pre_fix}(ii) implies that
		\begin{equation*}
		H\in\{3^8.\POmega_8^-(3).2, (3^2{:}2\times \POmega_8^+(3)).S_4, \PGL_2(29)\}\subseteq\mathcal{A}_2.
		\end{equation*}
		In each case, one can check that $32\notin\Spec(H)$, which completes the proof of Theorem \ref{thm:spo}.
		
		The statement of Corollary \ref{cor:rho_1_An} for sporadic groups can be deduced easily.
	\end{proof}

	It might be difficult to determine all the large fixers of sporadic groups since the list of subgroups of some sporadic groups is not easy to work out computationally. However, we are able to establish Theorem \ref{thm:spo_rho_0}, which verifies Conjecture \ref{conj:rho_0_as} in this setting.
	
	\begin{proof}[Proof of Theorem \ref{thm:spo_rho_0}]
		It suffices to show that
		\begin{equation}\label{e:rho_0}
		2|K|^2 < |G|\cdot |H|
		\end{equation}
		for any fixer $K$ of $G$. Assume $K$ is a fixer of $G$ and let $\widehat{K}$ be a maximal overgroup of $K$ in $G$. Define
		\begin{equation*}
		|\widehat{K}|_{\pi(H)}:=\prod_{p\in \pi(H)}|\widehat{K}|_p,
		\end{equation*}
		which is the $\pi(H)$-part of $|\widehat{K}|$. By Lemma \ref{l:pre_fix}(ii), we see that $|K| \leqs |\widehat{K}|_{\pi(H)}$, so \eqref{e:rho_0} holds if
		\begin{equation}\label{e:spo_size}
		2(|\widehat{K}|_{\pi(H)})^2 < |G|\cdot |H|.
		\end{equation}
		One can check that this inequality holds unless $(G,H,\widehat{K})$ is one of the cases recorded in Table \ref{tab:spo_size}. Let us consider these remaining cases in turn.
		
		First assume $G = \mathbb{B}$, $H = \PGL_2(11)$ and $\widehat{K} = 2^{1+22}.\mathrm{Co}_2$. Here $|\widehat{K}|_{\pi(H)} = 2^{41}\cdot 3^6\cdot 5^3\cdot 11$, and we note that $\widehat{K}$ has two $G$-conjugacy classes of elements of order $5$, one of which does not meet $H$. This shows that $|K|_5\leqs 5^2$, so
		\begin{equation*}
		2|K|^2\leqs 2^{83}\cdot 3^{12}\cdot 5^4\cdot 11^2 < |G|\cdot |H|.
		\end{equation*}
		
		For the other cases listed in Table \ref{tab:spo_size}, we note that $K \ne \widehat{K}$ by Theorem \ref{thm:spo}, so $K$ is contained in a maximal subgroup $\widetilde{K}$ of $\widehat{K}$, which can be obtained using {\sc Magma}. One can establish \eqref{e:rho_0} by repeating the argument above.
	\end{proof}
	
	{\small
		\begin{table}[h!]
			\begin{tabular}{@{}lll@{}}
				\toprule
				$G$ & $H$ & $\widehat{K}$ \\ \midrule
				$\Ma_{11}$&$S_5$&$\Ma_{10}$\\
				$\Ma_{12}$&$\PSL_2(11)$&$\Ma_{11}$\\
				$\mathrm{Co}_3$&$\Ma_{23}$&$\mathrm{McL}.2$\\
				$\mathrm{Co}_3$&$2\times \Ma_{12}$&$\mathrm{McL}.2$\\
				$\mathrm{Co}_3$&$A_4\times S_5$&$\mathrm{McL}.2$\\
				\bottomrule
			\end{tabular}
			\begin{tabular}{@{}lll@{}}
				\toprule
				$G$ & $H$ & $\widehat{K}$ \\ \midrule
				$\mathrm{Co}_2$&$\Ma_{23}$&$\PSU_6(2).2$\\
				$\mathrm{Fi}_{22}$&$\Ma_{12}$&$2.\PSU_6(2)$\\
				$\mathrm{Fi}_{22}.2$&$\Ma_{12}.2$&$2.\PSU_6(2).2$\\
				$\mathrm{Fi}_{24}'$&$\PGL_2(13)$&$\mathrm{Fi}_{23}$\\
				$\mathbb{B}$&$\PGL_2(11)$&$2^{1+22}.\mathrm{Co}_2$\\
				\bottomrule
			\end{tabular}
			\caption{The triples $(G,H,\widehat{K})$ such that \eqref{e:spo_size} does not hold} \label{tab:spo_size}
		\end{table}
	}

	\section{Alternating and symmetric groups}

	\label{s:max}

	In this final section, we consider the groups $G$ with $\soc(G) = A_n$ for some $n\geqs 5$. We seek to determine the maximal subgroups of $G$ which are large fixers.
	
	\begin{prop}
		\label{p:max_An_24}
		The conclusion to Theorem \ref{thm:alt} holds for $n\leqs 24$.
	\end{prop}
	
	\begin{proof}
		This can be done with the aid of {\sc Magma}.
	\end{proof}
	
	Thus, in order to prove Theorem \ref{thm:alt}, we may assume $n\geqs 25$. Note that $H$ is naturally of one of the following three cases:
	
	\begin{itemize}\addtolength{\itemsep}{0.2\baselineskip}
		\item[{\rm (a)}] $H$ is intransitive on $[n]$, and thus $H$ is a setwise stabiliser of a $k$-subset of $[n]$ with $1\leqs k < n/2$;
		\item[{\rm (b)}] $H$ is imprimitive on $[n]$, and thus $H$ is the stabiliser of a partition of $[n]$ into $b$ parts, with each part of size $a$ (so $ab = n$);
		\item[{\rm (c)}] $H$ is primitive on $[n]$.
	\end{itemize}
	
	Now we compare the orders of maximal subgroups of $G$. To begin with, we record a classical result considering the sizes of primitive groups, which is \cite[Corollary 1.2]{M_prim_order}.
	
	\begin{lem}
		\label{l:max_size_prim}
		Let $X$ be a primitive group of degree $m\geqs 25$. Then $|X| < 2^m$ if $X\notin\{A_m,S_m\}$.
	\end{lem}
	
	The following records elementary bounds on factorials, where $e$ is the exponential constant.
	
	\begin{lem}
		\label{l:max_factorial}
		For any integer $m$, we have
		\begin{equation*}
		\frac{m^m}{e^{m-1}}\leqs m!\leqs \frac{m^{m+1}}{e^{m-1}}.
		\end{equation*}
	\end{lem}
	
	Using this bound, we are able to bound on the sizes of imprimitive groups.
	
	\begin{lem}
		\label{l:max_size_imprim}
		Suppose $n\geqs 25$ and let $X$ be a maximal subgroup of $G$ which acts imprimitively on $[n]$. Then
		\begin{equation}
		\label{e:max_imprim_m}
		2^n\leqs |X|\leqs 2\cdot \left(\left\lceil\frac{n}{2}\right\rceil!\right)^2,
		\end{equation}
		with the second equality holds if and only if $X = S_{n/2}\wr S_2$. Moreover, if $X$ does not fix a partition of $[n]$ into $2$ parts, then
		\begin{equation}\label{e:max_imprim_3_parts}
		|X| \leqs 6\cdot \left(\left\lceil\frac{n}{3}\right\rceil!\right)^3.
		\end{equation}
	\end{lem}
	
	\begin{proof}
		First note that
		\begin{equation*}
		\frac{1}{2}\cdot (a!)^b\cdot b! = \frac{1}{2}\cdot|S_a\wr S_b|\leqs |X|\leqs |S_a\wr S_b| = (a!)^b\cdot b!,
		\end{equation*}
		where $ab = n$.
		
		Now we consider the first inequality of \eqref{e:max_imprim_m}. By Lemma \ref{l:max_factorial}, we have
		\begin{equation*}
		\frac{1}{2}\cdot (a!)^b\cdot b! \geqs \frac{1}{2}\cdot\left(\frac{a^a}{e^{a-1}}\right)^b\cdot \frac{b^b}{e^{b-1}} = \frac{1}{2}\cdot \frac{n^n}{e^{n-1}}\cdot b^{b-n}.
		\end{equation*}
		Hence, to prove that $|X|\geqs 2^n$, it suffices to show that
		\begin{equation*}
		\left(\frac{n}{2e}\right)^n > b^{n-b},
		\end{equation*}
		which is equivalent to
		\begin{equation}
		\label{e:max_imprim_m_1}
		\frac{n}{2e} > \left(\frac{n}{a}\right)^{1-1/a}
		\end{equation}
		since $ab = n$.	Note that this clearly holds for $a > 2e > 4$, and so we only need to consider the cases where $a \in\{2,3,4,5\}$. One can check that \eqref{e:max_imprim_m_1} holds for each case in turn.
		
		Now we establish the second inequality in \eqref{e:max_imprim_m}. Using the bounds in Lemma \ref{l:max_factorial}, we see that
		\begin{equation*}
		(a!)^b\cdot b!\leqs \left(\frac{a^{a+1}}{e^{a-1}}\right)^b\cdot\frac{b^{b+1}}{e^{b-1}} = \frac{a^n\cdot a^b\cdot b^b\cdot b}{e^{n-1}} = \frac{n^n}{e^{n-1}}\cdot \frac{n^b}{b^{n-1}}
		\end{equation*}
		and
		\begin{equation*}
		2\cdot\left(\left\lceil\frac{n}{2}\right\rceil!\right)^2 \geqs 2\cdot\left(\frac{(n/2)^{n/2}}{e^{n/2-1}}\right)^2 = \frac{n^n}{e^{n-2}}\cdot \frac{1}{2^{n-1}}.
		\end{equation*}
		Thus, we only need to show that
		\begin{equation}
		\label{e:max_imprim_m_2}
		2^{n-1}\cdot e^{-1} \leqs \frac{b^{n-1}}{n^b}=:f(b).
		\end{equation}
		for $3\leqs b\leqs n/2$. To see this, note that
		\begin{equation*}
		f'(b) = \frac{b^{n-2}((n-1)-b\ln n)}{n^b},
		\end{equation*}
		which implies that $f(b)\geqs \min\{f(3),f(n/2)\}$. Now it is easy to show that \eqref{e:max_imprim_m_2} holds for $n\geqs 25$.
		
		To complete the proof, assume $X$ does not fix a partition of $[n]$ into $2$ parts, and we show that \eqref{e:max_imprim_3_parts} holds. We first note that \eqref{e:max_imprim_3_parts} holds for $n\leqs 44$, so we may assume $n\geqs 45$. Once again, we have $|X|\leqs (a!)^b\cdot b!$	with $ab = n$ and $b\geqs 3$. By arguing as above, it suffices to show that
		\begin{equation*}
		f(b)\geqs \frac{3^n}{6e^2}
		\end{equation*}
		for $4\leqs b\leqs n/2$. This is easy to obtain for all $n\geqs 45$ by noting that $f(b)\geqs\min\{f(4),f(n/2)\}$ with the same reason as above.
	\end{proof}
	
	Now we are able to verify the statement of Theorem \ref{thm:alt} when $H$ is intransitive on $[n]$.
	
	\begin{prop}
		\label{p:max_An_K_intrans}
		Suppose $n\geqs 25$ and $H$ is a setwise stabiliser of a $k$-subset of $[n]$ with $1\leqs k < n/2$. Then for any maximal subgroup $K$ of $G$, $K$ is not a large fixer of $G$ if $K\not\cong H$.
	\end{prop}
	
	\begin{proof}
		For convenience, we assume $G = S_n$, as the proof for $G = A_n$ is the same. First note that
		\begin{equation*}
		|H|\geqs k!(n-k)!\geqs  \left\lfloor\frac{n}{2}\right\rfloor!\cdot \left\lceil \frac{n}{2}\right\rceil!.
		\end{equation*}
		Now we use the bounds given in Lemma \ref{l:max_factorial}, so
		\begin{equation*}
		\left\lfloor\frac{n}{2}\right\rfloor!\cdot \left\lceil \frac{n}{2}\right\rceil! \geqs \left(\frac{(n/2)^{n/2}}{e^{n/2-1}}\right)^2 = \frac{(n/2)^n}{e^{n-2}}
		\end{equation*}
		and
		\begin{equation*}
		6\cdot \left(\left\lceil \frac{n}{3}\right\rceil!\right)^3\leqs 6n^3\cdot\left(\frac{(n/3)^{n/3+1}}{e^{n/3-1}}\right)^3 = 6n^3\cdot \frac{(n/3)^{n+3}}{e^{n-3}}.
		\end{equation*}
		With these bounds, it is routine to check that
		\begin{equation*}
		\left\lfloor\frac{n}{2}\right\rfloor!\cdot \left\lceil \frac{n}{2}\right\rceil! > 6\cdot \left(\left\lceil \frac{n}{3}\right\rceil!\right)^3
		\end{equation*}
		for all $n\geqs 60$ by arguing as in the proof of Lemma \ref{l:max_size_imprim}, so if $K$ is transitive on $[n]$ then $|K|\geqs |H|$ only if $K = S_{n/2}\wr S_2$ (one can check that this also holds if $25\leqs n\leqs 59$). Now if $K = S_{n/2}\wr S_2$ then there exists an element of $K$ of cycle type $[n/2,n/2]$, which does not fix any $k$-subset. Hence, $K = S_{n/2}\wr S_2$ is not a fixer of $G$.
		
		This implies that $K$ is intransitive on $[n]$, and thus $K$ fixes an $\ell$-subset for some $1\leqs \ell < n/2$. Note that
		\begin{equation*}
		|K| = \ell!(n-\ell)! \geqs k!(n-k)! = |H|
		\end{equation*}
		if and only if $\ell \leqs k$. Since $K\cong H$ if $\ell = k$, we may assume $\ell < k$ and
		\begin{equation*}
		K = \mathrm{Sym}(\{1,\dots,\ell\})\times \mathrm{Sym}(\{\ell+1,\dots,n\}).
		\end{equation*}
		If $\ell\geqs 2$, then $k > 2$ and  $x = (1,2)^\delta(\ell+1,\dots,n)$ does not fix a $k$-subset of $[n]$, where $\delta = 1$ if $n-\ell$ is even, and $\delta = 0$ otherwise (so $x$ is an even permutation), and thus $K$ is not a fixer of $G$.
		
		Finally, assume $\ell = 1$. If $k \geqs 3$ then either $(2,\dots,n)$ or $(3,\dots,n)$ is an even permutation, and we see that each of these elements does not fix a $k$-subset. Now assume $k = 2$. If $n$ is odd, then we take $x = (2,3,4)(5,\dots,n)\in K$, while if $n$ is even, then take $x = (2,3,4)(5,6,7,8)(9,\dots,n)\in K$. In either case, $x$ is an even permutation that does not fix any $2$-subset. This completes the proof.
	\end{proof}

	To complete the proof of Theorem \ref{thm:alt}, we turn to the case where $H$ is imprimitive on $[n]$.
	
	\begin{lem}
		\label{l:max_An_K_imprim_H_intrans}
		Suppose $H$ is imprimitive on $[n]$, and let $K$ be a maximal subgroup of $G$ which is intransitive on $[n]$. Then $K$ is not a fixer of $G$.
	\end{lem}
	
	\begin{proof}
		Suppose $H\cong (S_a\wr S_b)\cap G$ and $K\cong (S_k\times S_{n-k})\cap G$ is a fixer of $G$, where $a,b\geqs 2$ and $k< n/2$. Let $m\in\{n-k,n-k-1\}$ be an odd number, noting that $K$ has an $m$-cycle and $a\leqs m$. Then $H$ has an $m$-cycle since $K$ is a fixer of $G$, so we have $a\mid m$. If $a\leqs n/3$, then $a$ divides $m-2$ as $K$ has an $(m-2)$-cycle (note that $a \leqs m-2$), which yields $a = 1$ since $(m,m-2) = 1$, a contradiction. Now if $a = n/2$, then $m =n/2$, $k = n/2-1$ and $n\equiv 2\pmod 4$. In this setting, $K$ has an element of cycle type $[n/2-1,n/2+1]$, which is not $G$-conjugate to any element in $H$.
	\end{proof}

	\begin{lem}
		\label{l:max_An_K_imprim_H_imprim}
		Suppose $H$ is imprimitive on $[n]$, and let $K$ be a maximal subgroup of $G$ which is imprimitive on $[n]$. If $K\not\cong H$, then $K$ is not a fixer of $G$.
	\end{lem}
	
	\begin{proof}
		Suppose $H\cong (S_a\wr S_b)\cap G$ and $K\cong (S_{a'}\wr S_{b'})\cap G$, where $n = ab = a'b'$ and $a\ne a'$. We will prove that $K$ is not a fixer of $G$, and we divide the proof into several cases.
		
		\vs
		
		\noindent \textit{Case 1. $a<a'$ and $a$ is even.}
		
		\vs
		
		In this case, $K$ has an $(a+1)$-cycle. By arguing as in the proof of Lemma \ref{l:max_An_K_imprim_H_intrans}, we see that $a\mid a+1$, which is impossible.
		
		\vs
		
		\noindent \textit{Case 2. $a+1<a'$ and $a$ is odd.}
		
		\vs
		
		Note that $K$ has an $(a+2)$-cycle. Thus, with the same reason as above, we have $a\mid a+2$. But this yields $a = 2$, which is incompatible with our assumption.
		
		\vs
		
		\noindent \textit{Case 3. $a+1=a'$ and $a$ is odd.}
		
		\vs
		
		Here we see that $K$ has an element of cycle type $[(a')^2,1^{n-2a'}]$, which is not $G$-conjugate to any element in $H$.
		
		\vs
		
		\noindent \textit{Case 4. $a > a'$ and $a'$ is odd.}
		
		\vs
		
		We have $3\leqs a' < a\leqs n/2$ in this setting. First assume $n-a'$ is odd. Then we have $a\mid n-a'$ as $K$ contains an $(n-a')$-cycle, which yields $a\mid a'$ as $a\mid n$. However, this is incompatible to our assumption $a > a'$. Now assume $n-a'$ is even, so both $n$ and $n-2a'$ are odd. Note that $K$ has an $(n-2a')$-cycle and $n-2a' > a$, and thus $a\mid n-2a'$. But this forces $a\mid 2a'$ and hence $a = 2a'$ is the only possibility, which gives a contradiction because $n$ is odd and $a$ is even.
		
		\vs
		
		\noindent \textit{Case 5. $a > a'$ and $a'$ is even.}
		
		\vs
		
		In this final case, we note that $K$ has an element of cycle type $[(n-a')^1,2^1, 1^{a-2}]$, and it is routine to check that this element is not $G$-conjugate to any element in $H$.
	\end{proof}
	
	\begin{prop}
		\label{p:max_An_K_imprim}
		Suppose $H$ is imprimitive on $[n]$. Then for any maximal subgroup $K$ of $G$ with $K\not\cong H$, $K$ is not a large fixer of $G$.
	\end{prop}
	
	\begin{proof}
		Note by Lemmas \ref{l:max_size_prim} and \ref{l:max_size_imprim} that $|K| < |H|$ if $K$ is primitive on $[n]$. Now combine Lemmas \ref{l:max_An_K_imprim_H_intrans} and \ref{l:max_An_K_imprim_H_imprim}.
	\end{proof}

	The proof of Theorem \ref{thm:alt} is complete by Propositions \ref{p:max_An_24}, \ref{p:max_An_K_intrans} and \ref{p:max_An_K_imprim}.

	Finally, we consider the cases where $H$ is primitive on $[n]$.
	
	\begin{lem}
		\label{l:max_H_prim_K_imprim}
		Suppose $H$ is a primitive on $[n]$ and $K$ is a maximal subgroup of $G$ which is not primitive on $[n]$. Then $K$ is not a fixer of $G$.
	\end{lem}
	
	\begin{proof}
		Note that $K$ contains a $3$-cycle or a product of two disjoint transpositions. Then apply the classical theorem of Jordan \cite{J_cycle} and a theorem of Manning \cite{M_transpo}.
	\end{proof}
	
	\begin{cor}
		\label{c:rho_1_An}
		The conclusion to Corollary \ref{cor:rho_1_An} holds when $\soc(G) = A_n$.
	\end{cor}
	
	\begin{proof}
		In view of Theorem \ref{thm:alt}, we may assume $H$ is primitive on $[n]$, and by Lemma \ref{l:max_H_prim_K_imprim}, we only need to consider the case where $K$ is primitive. The groups with $n < 24$ can be handled easily, and we see that the required bound holds for $n\geqs 25$ by combining the inequalities $|H|\geqs n > 5/2$ and $|K|^2 < 4^n  < n!/2\leqs G$.
	\end{proof}

Therefore, the proof of Corollary \ref{cor:rho_1_An} is complete by combining Proposition \ref{p:spo} and Corollary \ref{c:rho_1_An}.
	
	In view of Lemma \ref{l:max_H_prim_K_imprim}, we only need to consider the cases where both $H$ and $K$ are primitive on $[n]$ to study the maximal subgroups which occur as large fixers of $G$. Note that $H$ is of one of the following types:
	\begin{itemize}\addtolength{\itemsep}{0.2\baselineskip}
		\item[{\rm HA:}] $H = \mathrm{AGL}_d(r)\cap G$, where $n = r^d$ for some prime $r$.
		\item[{\rm PA:}] $H = (S_a\wr S_b)\cap G$, where $n = a^b$ with $a\geqs 5$ and $b\geqs 2$.
		\item[{\rm SD:}] $H = T^k.(\Out(T)\times S_k)$ for some non-abelian simple group $T$, where $n = |T|^{k-1}$.
		\item[{\rm AS:}] $H$ is an almost simple primitive group on $[n]$.
	\end{itemize}
	
	We will prove that if $K$ is of type HA or PA and $H\not\cong K$, then $K$ is not a large fixer of $G$ (see Lemmas \ref{l:max_K=HA} and \ref{l:max_K=PA} below). To do this, we will repeatedly use Corollary \ref{c:pre_min_deg} on the minimal degrees of $K$ and $H$ on $[n]$, and we first record \cite[Theorem 4]{BG_min_deg} on the study of minimal degrees of primitive groups.
	
	\begin{thm}
		\label{t:BG_min_deg}
		Suppose $X\leqs \mathrm{Sym}(\Delta)$ is a primitive permutation group of degree $n$ with point stabiliser $Y$. Then either $\mu_\Delta(X)\geqs 2n/3$, or up to isomorphism, one of the following holds:
		\begin{itemize}\addtolength{\itemsep}{0.2\baselineskip}
			\item[{\rm (i)}] $X = S_m$ or $A_m$ acting on $\ell$-subsets of $[m]$ with $1\leqs \ell < m/2$ and $n = {m\choose \ell}$.
			\item[{\rm (ii)}] $X = S_m$, $Y = S_{m/2}\wr S_2$, $n = \frac{m!}{2(m/2)!^2}$ and
			\begin{equation*}
			\mu_\Delta(X) = \frac{1}{2}\left(1+\frac{1}{m-1}\right)n.
			\end{equation*}
			\item[{\rm (iii)}] $X = \Ma_{22}.2$, $Y = \PSigmaL_3(4)$, $n = 22$ and $\mu_\Delta(X) = 14$.
			\item[{\rm (iv)}] $(X,Y,\mu_\Delta(X))$ is one of the cases listed in \cite[Table 2]{BG_min_deg}, where $X$ is almost simple.
			\item[{\rm (v)}] $X = V{:}Y$ is an affine group with $\Delta = V = \mathbb{F}_2^d$, $Y\leqs \GL_d(2)$ contains a transvection and $\mu_\Delta(X) = 2^{d-1} = n/2$.
			\item[{\rm (vi)}] $X \leqs L\wr S_k$ with $\Delta = \Sigma^k$, where $k\geqs 2$ and $L\leqs\mathrm{Sym}(\Sigma)$ is one of the almost simple primitive groups in parts (i)--(iv).
		\end{itemize}
	\end{thm}

	\begin{lem}
		\label{l:max_K=HA}
		Suppose $H$ is a primitive group on $[n]$ and $K\not\cong H$ is a maximal subgroup of $G$ of type HA. Then $K$ is not a large fixer of $G$.
	\end{lem}
	
	\begin{proof}
		Suppose $K = \mathrm{AGL}_d(r)\cap G$ is a large fixer of $G$. First note that $K$ has an element $g$ of cycle type
		\begin{equation*}
		[(r^d-1)/2,(r^d-1)/2,1],\
		\end{equation*}
		which is $G$-conjugate to an element in $H$. This implies that $H$ is of type HA or AS by \cite[Theorem 1.1]{GMPS_cycle}. If $H$ is of type HA, then we have $K\cong H$.
		
		It suffices to consider the case where $H$ is of type AS. Note that $H$ contains a $G$-conjugate of $g$. Thus, by inspecting the groups listed in \cite[Theorem 1.1]{GMPS_cycle}, which classifies the finite primitive permutation groups containing a permutation having at most four cycles, we see that either $\soc(H) = \PSL_2(p)$ and $n = p+1$ for some prime $p$, or $\soc(H) = \PSL_m(2)$ and $n = 2^m-1$.
		
		In the former case,	$n = 2^d$ is even, so by Theorem \ref{t:BG_min_deg}, we have
		\begin{equation*}
		\mu_{[n]}(K) = n/2 < 2n/3\leqs\mu_{[n]}(H),
		\end{equation*}
		which is incompatible with Corollary \ref{c:pre_min_deg}. Finally, in the latter case, we have $2^m - 1 = r^d$. Now Catalan's conjecture, which is proved true in \cite{M_Catalan}, implies that $d = 1$. However,
		\begin{equation*}
		|K|\leqs r(r-1) = (2^m-1)(2^m-2) < |\PSL_m(2)|/2 \leqs |H|,
		\end{equation*}
		a contradiction to the assumption that $K$ is a large fixer.
	\end{proof}
	
	\begin{lem}
		\label{l:max_K=PA}
		Suppose $H$ is a primitive group on $[n]$ and $K\not\cong H$ is a maximal subgroup of $G$ of type PA. Then $K$ is not a large fixer of $G$.
	\end{lem}
	
	\begin{proof}
		We may write
		\begin{equation*}
		[n] = \Delta^b = \Delta_1\times\cdots\times \Delta_b
		\end{equation*}
		with $|\Delta_i| = a$, and
		\begin{equation*}
		K = ((\mathrm{Sym}(\Delta_1)\times\cdots\times\mathrm{Sym}(\Delta_b)){:}\mathrm{Sym}([b]))\cap G.
		\end{equation*}
		The action of $K$ on $[n]$ is given by
		\begin{equation}\label{e:prod}
		\left(\delta_1, \dots, \delta_b\right)^{\left(x_1, \dots, x_b\right) \sigma}=\left(\delta_{1^{\sigma^{-1}}}^{x_{1^{\sigma{-1}}}}, \dots, \delta_{b^{\sigma^{-1}}}^{x_{b^{\sigma-1}}}\right)
		\end{equation}
		for any $\delta_i\in\Delta_i$, $x_i\in \mathrm{Sym}(\Delta_i)$ and $\sigma\in \mathrm{Sym}([b])$.
		
		Let $x = (x_1,1,\dots,1)$, where $x_1$ is a $3$-cycle in $\mathrm{Sym}(\Delta_1)$. Obviously, $x$ is an even permutation in $\mathrm{Sym}([n])$, so $x\in K$. Note that
		\begin{equation*}
		\left(\delta_1, \dots, \delta_b\right)^x= (\delta_1^{x_1},\delta_2,\dots,\delta_b)
		\end{equation*}
		for any $(\delta_1,\dots,\delta_b)\in [n]$. This implies that
		\begin{equation*}
		\fpr_{[n]}(x) = \frac{(a-3)a^{b-1}}{a^b} = \frac{a-3}{a}
		\end{equation*}
		and so $\mu_{[n]}(K)\leqs 3a^{b-1} = 3n/a$ (in fact, if $G = S_n$ then $\mu_{[n]}(K) = 2n/a$ as noted in \cite[p. 130]{GM_min_deg}). In particular, we have $\mu_{[n]}(K) < 2n/3$, and by Corollary \ref{c:pre_min_deg} we have $\mu_{[n]}(K)\geqs \mu_{[n]}(H)$. Therefore, $H$ is one of the primitive groups listed in Theorem \ref{t:BG_min_deg}.
		
		First assume case (ii) of Theorem \ref{t:BG_min_deg}, so $n = \frac{m!}{2(m/2)!^2}$ and
		\begin{equation*}
		\mu_{[n]}(H) = \frac{1}{2}\left(1+\frac{1}{m-1}\right)n.
		\end{equation*}
		Thus, $\mu_{[n]}(K)\geqs\mu_{[n]}(H)$ only if $a = 5$. However, $2n = {m\choose m/2}$ is clearly not a $5$-power, so no case with $a = 5$ arises. We do not need to consider case (iii) of Theorem \ref{t:BG_min_deg} since $n\geqs 25$ is assumed. By inspecting \cite[Table 2]{BG_min_deg} and using Catalan's conjecture \cite{M_Catalan}, one can see that either $n = a^b$ does not have an integer solution with $a\geqs 5$ and $b\geqs 2$, or $\mu_{[n]}(H) > 3n/5$, so case (iv) of Theorem \ref{t:BG_min_deg} does not arise. For example, if $\soc(H) = \PSL_d(2)$ with its action on $1$-subspaces of $\mathbb{F}_d^2$, then $n = 2^d-1$, and as noted above, there is no integer solution to $2^d-a^b = 1$ with $b\geqs 2$. If $H$ is of case (v) of Theorem \ref{t:BG_min_deg}, then $\mu_{[n]}(H) = n/2$ and $a$ is a $2$-power, which implies that $a\geqs 8$ and $\mu_{[n]}(K)\leqs 3n/8$, so this is incompatible to Corollary \ref{c:pre_min_deg}.
		
		Now let us consider case (i) of Theorem \ref{t:BG_min_deg}. Here we have $\soc(H) = A_m$ and
		\begin{equation}
		\label{e:int_sol}
		a^b = {m\choose \ell},
		\end{equation}
		where $2\leqs\ell\leqs m/2$. By \cite[Theorem 2]{G_dio}, we see that apart from the case $b = \ell = 2$, the only integer solution to \eqref{e:int_sol} is $(a,b,m,\ell) = (140,2,50,3)$. For the latter case, we have $53\in\pi(K)\setminus\pi(H)$, which is incompatible with Lemma \ref{l:pre_fix}(ii). Thus, we may now assume $b = \ell = 2$, so $a^2 = {m\choose 2}$. Recall that $x = (x_1,1,\dots,1)$, where $x_1$ is a $3$-cycle in $\mathrm{Sym}(\Delta_1)$. As noted above, we have $\fpr_{[n]}(x) = 1-3/a$. Moreover, \cite[Proposition 3.4]{BG_min_deg} shows that
		\begin{equation*}
		\fpr_{[n]}(y)\leqs 1-\frac{6(m-2)}{m(m-1)}
		\end{equation*}
		for any element $y\in H$ of order $3$.
		In view of Lemma \ref{l:pre_fpr}, it suffices to show that
		\begin{equation*}
		\frac{3}{a} < \frac{6(m-2)}{m(m-1)}.
		\end{equation*}
		This is equivalent to $m-2 > a$, which is obvious since $a^2 = m(m-1)/2$.
		
		For the remained of the proof, let us turn to case (vi) of Theorem \ref{t:BG_min_deg}, where $H$ is also a maximal subgroup of $G$ of type PA, which requires a more detailed analysis. First note that there exists an integer $c$ such that $a = c^\ell$ and $H = (S_{c^m}\wr S_{d/m})\cap G$ for some integers $\ell$ and $m$ with $\ell\ne m$. To be precise, we may identify
		\begin{equation*}
		[n] = \Gamma^{d/m} = \Gamma_1\times \cdots\times\Gamma_{d/m},
		\end{equation*}
		where $|\Gamma_j| = c^m$, and write
		\begin{equation*}
		H = ((\mathrm{Sym}(\Gamma_1)\times\cdots\times\mathrm{Sym}(\Gamma_{d/m})){:}\mathrm{Sym}([d/m]))\cap G,
		\end{equation*}
		with the similar componentwise action as given in \eqref{e:prod}. We divide the proof into five cases.
		
		\vs
		
		\noindent \emph{Case 1. $m<\ell$}
		
		\vs
		
		In this case, we have
		\begin{equation*}
		\mu_{[n]}(H)\geqs 2n/c^m > 3n/c^\ell \geqs \mu_{[n]}(K)
		\end{equation*}
		as discussed above, which is incompatible with Corollary \ref{c:pre_min_deg}.
		
		\vs
		
		\noindent \emph{Case 2. $\ell = 2$, $m\geqs 3$ is odd and $c+1$ is a $2$-power}
		
		\vs
		
		First assume $c\neq 3$, noting that $8$ divides $c+1.$ Let $k=(\frac{d}{2})_2$ be the greatest $2$-power dividing $\frac{d}{2}$. So $(\frac{d}{m})_2 = 2k$ as $m$ is odd. Let $\sigma=\sigma_1\sigma_2\dots \sigma_{\frac{d}{2k}}\in \mathrm{Sym}([d/2])$, where $\sigma_i=(ik-k+1, ik-k+2,\dots, ik)$ is a $k$-cycle, and let $\tau\in \Sym(\Delta)$ be a product of $c-1$ disjoint $(c+1)$-cycles.  We claim that
		\begin{equation*}
		x=(\tau, \underbrace{1, \dots, 1}_{k-1}, \dots, \tau, \underbrace{1, \dots, 1}_{k-1})\sigma\in K
		\end{equation*} 
		is not $G$-conjugate to any element in $H$.
		
		To see this, first note that $x$ is the product of $\frac{c^d-1}{k(c+1)}$ disjoint $k(c+1)$-cycles. Suppose $h\in H$ is $G$-conjugate to $x$. Then by \cite[Proposition 2.1]{CH_conju}, we may assume
		\begin{equation*}
		h=(h_1, \underbrace{1, \dots, 1}_{s_1-1},h_2, \underbrace{1,\dots, 1}_{s_2-1}, ,\dots, h_t, \underbrace{1,\dots, 1}_{s_t-1})\sigma' \in H
		\end{equation*}
		for some integers $s_1,\dots,s_t\geqs 1$, where $\sigma' = \sigma_1\sigma_2\cdots\sigma_t$ is such that
		\begin{equation*}
		\sigma_i = (s_1+\cdots+s_{i-1}+1,\dots,s_1+\cdots+s_i).
		\end{equation*}
		As $|h| = |x| = k(c+1)$ is a power of $2$, we see that $|h_i|$ and $|\sigma_i|$ are $2$-powers. Moreover, we note that $|\fix_{[n]}(h)| = |\fix_{[n]}(x)| = 1$ and
		\begin{equation*}
		|\fix_{[n]}(h)| = |\fix_{\Gamma}(h_1)|^{s_1}\cdots |\fix_{\Gamma}(h_t)|^{s_t}.
		\end{equation*}
		This implies that $|\fix_\Gamma(h_i)| = 1$. 
		One can also observe that there exists $1\leqs i \leqs t$ such that $|\sigma_i|$ divides $2k$, otherwise $4k$ divides $|\sigma_i|$ for each $1\leqs i \leqs t$ and hence $$|\sigma_1|+|\sigma_2|+\dots +|\sigma_t|= \frac{d}{m}$$ is divisible by $4k,$ which is incompatible to $(\frac{d}{m})_2=2k$. 
		Note that $h$ contains a cycle of length $2ok$ for some integer $o$ if $h_i$ contains a cycle of length $o$. 
		Thus, $h_i$ is a product of $\frac{2(c^m-1)}{c+1}$ disjoint $\frac{c+1}{2}$-cycles. This is impossible, since $c^m-1$ is indivisible by $4$.
		
		We next assume $c=3$. It follows from the main theorem of \cite{N_prime} that either $d=6$ or there exists a prime $p$ such that $\frac{d}{m}<p<\frac{d}{2}$. For the former, it is easy to check that $|H|>|K|.$ 
		Now assume the latter case and let $\sigma=(1,2,\dots, p)\in S_{\frac{d}{2}}\leqs K$. We claim that $\sigma$ is not $G$-conjugate to any element of $H.$
		
		To verify the claim, we first observe that $\sigma$ fixes exactly $3^{d-2p+2}$ points of $[n]$.  Assume $\sigma$ is $G$-conjugate to an element $h\in H$. Then $h\in \Sym(\Gamma_1)\times \cdots \times \Sym(\Gamma_{d/m})$ since $p > d/m$. Thus, we may write 
		\begin{equation*}
		h=(h_1,h_2,\dots, h_{d/m}),
		\end{equation*}
		where
		\begin{equation*}
		h_1^p=h_2^p=\dots h_{d/m}^p=1.
		\end{equation*}
		Since $|h| = |\sigma| = p$, there exists $1\leqs j \leqs \frac{d}{m}$ such that $|h_j|=p.$
		Thus, $|\fix_{\Gamma}(h_j)|$ is coprime to $3,$ a contradiction to $|\fix_{\Gamma}(h_j)|$ dividing $|\fix_{[n]}(h)|=3^{d-2p+2}.$
		
		\vs
		
		\noindent \emph{Case 3. $m > \ell$, $\ell$ does not divide $m$, and $(c,m,\ell)$ does not satisfy Case 2}
		
		\vs
		
		By Zsigmondy's theorem \cite{Z_ppd}, there exists a prime divisor $r$ of $c^\ell-1$ that does not divide $c^m-1$. Let $\tau$ be an element in $\mathrm{Sym}(\Delta)$ of order $r$ that fixes a unique point of $\Delta$, and let $x = (\tau,\dots,\tau)\in K$, noting that $x$ fixes a unique point in $[n]$. We claim that $x$ is not $G$-conjugate to any element in $H$.
		
		Suppose $h\in H$ is $G$-conjugate to $x$. Again, by \cite[Proposition 2.1]{CH_conju}, we may assume
		\begin{equation*}
		h = (h_1,\underbrace{1,\dots, 1}_{r-1},h_2,\underbrace{1, \dots, 1}_{r-1},\dots, h_{t'},\underbrace{1, \dots, 1}_{r-1},h_{t'+1},h_{t'+2},\dots, h_{t})\sigma\in H
		\end{equation*}
		for some $t$ and $t'$, where $\sigma_i = (ri-r+1,\dots,ri)$ and $\sigma = \sigma_1\cdots\sigma_{t'}$. Since $h^r = 1$, we have
		\begin{equation*}
		h_1 = \cdots = h_{t'} = h_{t'+1}^r = \cdots = h_t^r = 1.
		\end{equation*}
		Hence, $t' = 0$ since $h$ has a unique fixed point in $[n]$, and moreover, $h_i$ fixes a unique point in $\Delta$ if $t'+1\leqs i\leqs t$, which is impossible since $r$ does not divide $c^m-1$.
		
		Now it suffices to consider the case where $\ell$ divides $m$, and thus from now we write $m = \ell t$ for some $t\geqs 2$ (so $c^m = a^t$ and $d/m = b/t$). 
		
		\vs
		
		\noindent \emph{Case 4. $\ell$ divides $m$ and $b\geqs 2a$}
		
		\vs
		
		Let $r$ be a prime such that $b/2 < r < b$ (the existence follows from Bertrand postulate). Note that $r > b/2 \geqs a$. Let $x$ be an $r$-cycle in $\mathrm{Sym}([b])$. Then $x$ fixes exactly $a^{b-r+1}$ points in $[n]$. Suppose $h\in H$ is of order $r$, which fixes $a^{b-r+1}$ points in $[n]$. The assumption $b/2 < r$ yields $h\in S_{a^t}^{b/t}$, so we may write $h = (h_1,\dots,h_{b/t})$ with $h_i\in \mathrm{Sym}(\Gamma_i)$, noting that there exists $j$ such that $|h_j| = r$. Now $|\fix_{\Gamma_j}(h_j)| = a^t-kr$ for some positive integer $k$. However, $|\fix_{\Gamma_j}(h_j)|$ divides $|\fix_{[n]}(x)| = a^{b-r+1}$, whereas
		\begin{equation*}
		(a^t,a^t-kr) = (a^t,r) = 1
		\end{equation*}
		implies that $|\fix_{\Gamma_j}(h_j)|$ is coprime to $a$. This gives a contradiction, so $x$ is not $G$-conjugate to any element in $H$.
		
		\vs
		
		\noindent \emph{Case 5. $\ell$ divides $m$ and $b < 2a$}
		
		\vs
		
		In this final case, we will prove that $|H| > |K|$. To do this, first note that
		\begin{equation}\label{e:div}
		\frac{|H|}{|K|} = \frac{(a^t)!^{b/t}(\frac{b}{t})!}{(a!)^b\cdot b!} =\frac{\prod_{j=a+1}^{a^t}j^{b/t}}{(a!)^{b-b/t}\cdot \prod_{k=b/t+1}^{b} k} > \frac{\prod_{j=a+1}^{a^t}j^{b/t-1}\cdot \prod_{i=a+1}^{a^t-b+b/t}i}{(a!)^{b-b/t}}
		\end{equation}
		since $b< 2a <a^t$. Observe that the numerator of the last term of \eqref{e:div} is a product of $(a^t-a)b/t+b/t-b$ integers strictly larger than $a$, so it suffices to show that
		\begin{equation*}
		(a^t-a)b/t+b/t-b \geqs a(b-b/t).
		\end{equation*}
		This is given by the inequality $a^t+1\geqs t(a+1)$.
		
		We conclude the proof since all cases have been discussed.
	\end{proof}

	\section{The tables}
	
	\label{s:tab}
	
	Here we present Table \ref{tab:PSL2} and \ref{tab:sporadic}, which arise in the statements of Theorems \ref{thm:PSL2} and \ref{thm:spo}, respectively.
	
	{\small
		\begin{table}
			\[
			\begin{array}{lll} \hline
			\mbox{Type of $H$}&L_0&\mbox{Conditions}\\\hline
			\GL_2(q_0) & C_p^f{:}C_{\frac{q_0-1}{(2,q-1)}} & \mbox{$q = q_0^r$, $r$ odd, $r\ne p$, and $L = L^{\rm I}$ is as defined in \eqref{e:LI}}\\
			& & \mbox{$q = q_0^p$, $p\ne 2$, $q \ne 27$, $(p,|K:K_0|) = 1$, and $L = L^{\rm II}$ is as defined in \eqref{e:LII}}\\
			& C_p^{f-f/p}{:}C_{\frac{q_0-1}{2}} & \mbox{$q = q_0^p$, $p > 3$, $p$ divides $|K:K_0|$, and $L = L^{\rm III}$ is as defined in \eqref{e:LIII}}\\
			& C_2^f{:}C_{2^{f/2}+1} & \mbox{$p = 2$, $q = q_0^2$ and $|G:G_0|$ odd}\\
			\GL_1(q)\wr S_2 & C_2^f{:}C_{2^f-1} & \mbox{$p = 2$ and $|L:L_0|$ odd}\\
			& \SL_2(q^{1/2}) & \mbox{$p = 2$ and $f$ even}\\
			& A_4 & \mbox{$q = 13$ or $G = \PGL_2(7)$}\\
			& S_4 & \mbox{$q = 25$ and $G\leqs\PSigmaL_2(q)$}\\
			& A_5 & \mbox{$q = 31$, or $q\in\{16,61\}$ and $G = G_0$}\\
			\GL_1(q^2) & A_4 & \mbox{$q = 11$ or $G = \PSL_2(5)$}\\
			& S_4 & G = \PSL_2(23)\\
			& A_5 & \mbox{$q = 29$ or $G = \PSL_2(59)$}\\
			2^{1+2}_-.\mathrm{O}_2^-(2) & S_4 & \mbox{$G = G_0$ and $q = p\equiv \pm 1\pmod 8$}\\
			A_5 & A_5 & \mbox{$G = G_0$ and $q = p\equiv \pm 1\pmod{10}$}\\
			& & \mbox{$G = G_0$, $q = p^2$ and $3\ne p\equiv \pm 3\pmod{10}$}\\
			\hline
			\end{array}
			\]
			\caption{Maximal non-stable fixers of $G$ with $\soc(G) = \PSL_2(q)$}
			\label{tab:PSL2}
	\end{table}}
	
	\begin{rem}
		\label{r:tab:PSL2}
		Let us record some additional comments on Table \ref{tab:PSL2}.
		\begin{itemize}\addtolength{\itemsep}{0.2\baselineskip}
			\item[{\rm (i)}] The \textit{type} of $H$, recorded in the first column, gives an  approximate description of the structure of $H$, which is consistent with usage in \cite{KL_classical}.
			\item[{\rm (ii)}] We write $K_0$ and $L_0$ for the groups $K\cap G_0$ and $L\cap G_0$, respectively, where $G_0 = \soc(G)$.
			\item[{\rm (iii)}] In the second column, we record $L_0$ up to isomorphism. Note that there might be more than one conjugacy classes of subgroups $L$ satisfying the condition described in the table. But we remark that apart from the first two and the last three rows, if $K\leqs G$ is such that $K_0\cong L_0$ and $K/K_0\cong L/L_0$, then $K$ is a fixer of $G$. For example, see Lemma \ref{l:PSL2_d_III_rem} for an argument on the third row.
			\item[{\rm (iv)}] Let us turn to the first and second rows of Table \ref{tab:PSL2}. Then $H = \GL_2(q_0)$, and we have $L = L^{\rm I}$ defined in \eqref{e:LI} in the former case, while $L = L^{\rm II}$ as described in \eqref{e:LII} for the latter. It is worth noting that not every subgroup $K$ with $K_0\cong L_0$ and $K/K_0\cong L/L_0$ is isomorphic to $L$ (we refer the reader to Remark \ref{r:LI_LII} for more details).
			\item[{\rm (v)}] In the last three rows where $H$ is of type $2^{1+2}_-.\mathrm{O}_2^-(2)$ or $A_5$, we have $L = H^\delta$, where $\delta$ is a diagonal automorphism of $G_0$.
		\end{itemize}
	\end{rem}
	
	{\small
		\begin{table}[h!]
			\begin{tabular}{@{}lll@{}}
				\toprule
				$G$ & $H$ & $K$ \\ \midrule
				$\Ma_{11}$&$\GL_2(3)$&$\Ma_9{:}2$\\
				$\Ma_{22}$&$2^4.S_5$&$2^4.A_6$\\
				&$A_7$&$A_7$\\
				$\Ma_{23}$&$\PSigmaL_3(4)$&$2^4{:}A_7$\\
				&$2^4{:}A_7$&$\PSigmaL_3(4)$\\
				$\mathrm{Co}_3$&$\PSL_3(4).D_{12}$&$2^4.A_8$\\
				$\mathrm{J}_1$&$D_6\times D_{10}$&$2\times A_5$\\
				&$7{:}6$&$2^3{:}7{:}3$\\
				$\mathrm{J}_3$&$(3\times A_6){:}2$&$2^4{:}(3\times A_5)$\\
				&$\PSL_2(19)$&$\PSL_2(19)$\\
				$\mathrm{HS}$&$\PSU_3(5){:}2$&$\PSU_3(5){:}2$\\
				$\mathrm{Suz}$&$\PSL_3(3){:}2$&$\PSL_3(3){:}2$\\
				$\mathrm{He}$&$2^6{:}3.S_6$&$2^6{:}3.S_6$\\
				\bottomrule
			\end{tabular}
			\begin{tabular}{@{}lll@{}}
				\toprule
				$G$ & $H$ & $K$ \\ \midrule
				$\mathrm{McL}$&$\Ma_{22}$&$\Ma_{22}$\\
				&$2^4.A_7$&$2^4.A_7$\\
				&$2^4.A_7$&$\PSigmaL_3(4)$\\
				&$\PSigmaL_3(4)$&$2^4.A_7$\\
				$\mathrm{Ru}$&$2^{3+8}{:}\PSL_3(2)$&$2^6.\PSU_3(3).2$\\
				$\mathrm{Fi}_{22}$&$S_{10}$&$S_{10}$\\
				&$\Omega_7(3)$&$\Omega_7(3)$\\
				$\mathrm{HN}$&$\Ma_{12}{:}2$&$\Ma_{12}{:}2$\\
				$\mathrm{Fi}_{24}'$&$\PSU_3(3){:}2$&$\PSU_3(3){:}2$\\
				&$\PGL_2(13)$&$\PGL_2(13)$\\
				$\operatorname{O'N}$&$A_7$&$A_7$\\
				&$(3^2{:}4\times A_6).2$&$3^4{:}2^{1+4}.D_{10}$\\
				$\operatorname{O'N}.2$&$(3^2{:}4\times A_6).2.2$&$3^4{:}2^{1+4}.D_{10}.2$\\
				\bottomrule
			\end{tabular}
			\caption{Maximal subgroups of almost simple sporadic groups which are large fixers} \label{tab:sporadic}
		\end{table}
	}


\begin{thebibliography}{999}
		
		\bibitem{AM_EKR}
		B. Ahmadi and K. Meagher, \textit{The Erdős-Ko-Rado property for some permutation groups}, Australas. J. Combin. \textbf{61} (2015), 23–41.
		
		
		\bibitem{Magma}
		W. Bosma, J. Cannon and C. Playoust, \textit{The Magma algebra system. I. The user language}, J. Symb.
		Comput. \textbf{24} (1997), 235–265.
		
		\bibitem{BHR_classical}
		J.N. Bray, D.F. Holt and C.M. Roney-Dougal, \textit{The maximal subgroups of the low-dimensional finite classical groups}, London Math. Soc. Lecture Note Series, Vol. 407, Cambridge University Press, Cambridge, 2013.
		
		\bibitem{B_GAPCTL}
		T. Breuer, \emph{The \textsf{GAP} {C}haracter {T}able {L}ibrary, Version 1.3.1}, \textsf{GAP} package, \\ \texttt{http://www.math.rwth-aachen.de/\~{}Thomas.Breuer/ctbllib}, 2020.
		
		\bibitem{B_fpr_survey}
		T.C. Burness, \textit{Simple groups, fixed point ratios and applications},  in Local representation theory and
		simple groups, 267–322, EMS Ser. Lect. Math., Eur. Math. Soc., Zürich, 2018.
		
		\bibitem{BG_classical}
		T.C. Burness and M. Giudici, \textit{Classical groups, derangements and primes}, Australian Mathematical
		Society Lecture Series, Vol. 25, Cambridge University Press, Cambridge, 2016.
		
		\bibitem{BG_min_deg}
		T.C. Burness and R.M. Guralnick, \textit{Fixed point ratios for finite primitive groups and applications}, Adv. Math. \textbf{411} (2022), Paper No. 108778, 90 pp.
		
		\bibitem{CC_der}
		P.J. Cameron and A.M. Cohen, \textit{On the number of fixed point free elements in a permutation group}, Discrete Math. \textbf{106/107} (1992), 135–138.
		
		\bibitem{CH_conju}
		J.J. Cannon and D.F. Holt, \textit{Computing conjugacy class representatives in permutation groups}, J. Algebra \textbf{300} (2006), 213–222.
		
		\bibitem{CJ_cyclic}
		M. Costantini and E. Jabara, \textit{On finite groups in which cyclic subgroups of the same order are conjugate}, Comm. Algebra \textbf{37} (2009), 3966–3990.
		
		\bibitem{DLP_M}
		H. Dietrich, M. Lee and T. Popiel, \textit{The maximal subgroups of the Monster}, Adv. Math. \textbf{469} (2025), Paper No. 110214, 33 pp.
		
		\bibitem{EFP_EKR}
		D. Ellis, E. Friedgut and H. Pilpel, \textit{Intersecting families of permutations}, J. Amer. Math. Soc. \textbf{24} (2011), 649–682.
		
		\bibitem{EKR_EKR}
		P. Erdős, C. Ko and R. Rado, \textit{Intersection theorems for systems of finite sets}, Quart. J. Math. Oxford Ser. \textbf{12} (1961), 313–320.
		
		\bibitem{FD_EKR}
		P. Frankl and M. Deza, \textit{On the maximum number of permutations with given maximal or minimal distance}, J. Combinatorial Theory Ser. A \textbf{22} (1977), 352–360.
		
		\bibitem{FG_der_18}
		J. Fulman and R.M. Guralnick, \textit{Derangements in finite classical groups for actions related to extension field and imprimitive subgroups and the solution of the Boston-Shalev conjecture}, Trans. Amer. Math.
		Soc. \textbf{370} (2018), 4601–4622.
		
		\bibitem{FPS_der}
		M. Fusari, A. Previtali and P. Spiga, \textit{Cliques in derangement graphs for innately transitive groups}, J. Group Theory \textbf{27} (2024),  929–965.
		
		\bibitem{GM_EKR}
		C. Godsil and K. Meagher, \textit{Erdős-Ko-Rado theorems: algebraic approaches}, Cambridge Stud. Adv. Math., Vol. 149, Cambridge University Press, Cambridge, 2016.
		
		\bibitem{GLS_CFSG3}
		D. Gorenstein, R. Lyons and R. Solomon, \textit{The Classification of the Finite Simple Groups. Number 3}, Math. Surveys Monogr., Vol. 40, American Mathematical Society, Providence, RI., 1998.
		
		\bibitem{GMPS_cycle}
		S. Guest, J. Morris, C.E. Praeger and P. Spiga, \textit{Finite primitive permutation groups containing a permutation having at most four cycles}, J. Algebra \textbf{454} (2016), 233–251.
		
		\bibitem{GIP_der}
		R.M. Guralnick, I.M. Issacs and P. Spiga, \textit{On a relation between the rank and the proportion of derangements in finite transitive permutation groups}, J. Combin. Theory Ser. A \textbf{136} (2015), 198–200.
		
		\bibitem{GM_min_deg}
		R.M. Guralnick and K. Magaard, \textit{On the minimal degree of a primitive permutation group}, J. Algebra \textbf{207} (1998), 127–145.
		
		\bibitem{G_dio}
		K. Győry, \textit{On the Diophantine equation ${n\choose k} = x^l$}, Acta Arith. \textbf{80} (1997), 289–295.
		
		\bibitem{H_alg}
		I.N. Herstein, \textit{Topics in algebra} (Second edition), John Wiley \& Sons, New York, 1975.
		
		\bibitem{H_book}
		B. Huppert, \textit{Endliche Gruppen. I}, Die Grundlehren der mathematischen Wissenschaften, Band 134,
		Springer-Verlag, Berlin-New York, 1967.
		
		\bibitem{J_number}
		W. Jehne, \textit{Kronecker classes of algebraic number fields}, J. Number Theory \textbf{9} (1977), 279–320.
		
		\bibitem{J_cycle}
		C. Jordan, \textit{Sur la limite de transitivité des groupes non alternés}, Bull. Soc. Math. France \textbf{1} (1872/73), 40–71.
		
		\bibitem{J_der}
		C. Jordan, \textit{Recherches sur les substitutions}, J. Math. Pures Appl. (Liouville) \textbf{17} (1872), 351–367.
		
		\bibitem{KL_classical}
		P.B. Kleidman and M.W. Liebeck, \textit{The subgroup structure of the finite classical groups}, London Math. Soc. Lecture Note Series, Vol. 129, Cambridge University Press, Cambridge, 1990.
		
		\bibitem{K_arith}
		N. Klingen, \textit{Arithmetical similarities}, in Prime decomposition and finite group theory, The Clarendon Press, Oxford University Press, New York, 1998.
		
		\bibitem{LPSX_intersecting}
		C.H. Li, V.R.T. Pantangi, S.J. Song and Y.L. Xie, \textit{Intersecting subsets in finite permutation groups}, submitted (2024), arXiv:2403.17783.
		
		\bibitem{LPSX_EKR}
		L. Long, R. Plaza, P. Sin and Q. Xiang, \textit{Characterization of intersecting families of maximum size in $\PSL(2,q)$}, J. Combin. Theory Ser. A \textbf{157} (2018), 461–499.
		
		\bibitem{LPS_O'Nan-Scott}
		M.W. Liebeck, C.E. Praeger and J. Saxl, \textit{On the O'Nan-Scott theorem for finite primitive permutation groups}, J. Austral. Math. Soc. Ser. A \textbf{44} (1988), 389–396.
		
		\bibitem{M_transpo}
		W.A. Manning, \textit{The primitive groups of class $2p$ which contain a substitution of order $p$ and degree $2p$}, Trans. Amer. Math. Soc. \textbf{4} (1903), 351–357.
		
		\bibitem{M_prim_order}
		A. Maróti, \textit{On the orders of primitive groups}, J. Algebra \textbf{258} (2002), 631–640.
		
		\bibitem{MR_EKR}
		K. Meagher and A.S. Razafimahatratra, \textit{Some Erdős-Ko-Rado results for linear and affine groups of degree two}, Art Discrete Appl. Math. \textbf{6} (2023), Paper No. 1.05, 30 pp.
		
		\bibitem{MST_2-trans}
		K. Meagher, P. Spiga and P.H. Tiep, \textit{An Erdős-Ko-Rado theorem for finite $2$-transitive groups}, European J. Combin. \textbf{55} (2016), 100–118.
		
		\bibitem{M_Catalan}
		P. Mihăilescu, \textit{Primary cyclotomic units and a proof of Catalan's conjecture}, J. Reine Angew. Math. \textbf{572} (2004), 167–195.
		
		\bibitem{N_prime}
		J. Nagura, \textit{On the interval containing at least one prime number}, Proc. Japan Acad. \textbf{28} (1952), 177–181.
		
		
		\bibitem{P_Kro}
		C.E. Praeger, \textit{Covering subgroups of groups and Kronecker classes of fields}, J. Algebra \textbf{118} (1988), 455–463.
		
		\bibitem{PK_max_cyclic}
		V.V. Pylaev and N.F. Kuzennyi, \textit{Finite groups with a cyclic maximal subgroup}. Ukrainian Math. J. \textbf{28} (1976), 500–506.
		
		\bibitem{Sh_conj_AGammaL}
		H. Shahrtash, \textit{Conjugacy class sizes in affine semi-linear groups}, Adv. Group Theory Appl. \textbf{10} (2020), 67–81.
		
		\bibitem{S_06}
		P. Spiga, \textit{Permutation characters and fixed-point-free elements in permutation groups}, J. Algebra \textbf{299} (2006), 1–7.
		
		\bibitem{S_20}
		P. Spiga, \textit{On a conjecture on the permutation characters of finite primitive groups}, Bull. Aust. Math. Soc. \textbf{102} (2020), 77–90.
		
		\bibitem{W_WebAt}
		R.A. Wilson et al., \emph{A {W}orld-{W}ide-{W}eb {A}tlas of finite group representations},  \\ {\texttt{http://brauer.maths.qmul.ac.uk/Atlas/v3/}}
		
		\bibitem{Z_ppd}
		K. Zsigmondy, \textit{Zur Theorie der Potenzreste}, Monatsh. Math. Phys. \textbf{3} (1892), 265–284.
		
	\end{thebibliography}
\end{document}